\documentclass[12pt]{amsart}
\usepackage{amssymb,latexsym,amsmath,amscd,amsthm,graphicx}

\theoremstyle{plain}

\newtheorem{theorem}[subsection]{Theorem}

\newtheorem{corollary}[subsection]{Corollary}
\newtheorem{lemma}[subsection]{Lemma}

\theoremstyle{definition}
\newtheorem{prop}[subsection]{Proposition}
\newtheorem{cor}[subsection]{Corollary}

\newtheorem{remark}[subsection]{Remark}

\newcommand{\uu}{\cup}
\newcommand{\ii}{\cap}
\newcommand{\UU}{\bigcup}
\newcommand{\II}{\bigcap}

\newcommand{\sci}{\subset}
\newcommand{\es}{\emptyset}
\newcommand{\set}[1]{\{#1\}}

\newcommand{\gd}{\delta}
\renewcommand{\gg}{\gamma}
\newcommand{\gh}{\eta}

\newcommand{\gm}{\mu}

\newcommand{\go}{\omega}

\newcommand{\gq}{\theta}
\newcommand{\gs}{\sigma}
\newcommand{\gt}{\tau}

\newcommand{\tit}{\textit}

\newcommand{\C}[1]{\mathcal{#1}}
\newcommand{\D}[1]{\mathbb{#1}}
\newcommand{\F}[1]{\mathfrak{#1}}

\newcommand{\te}{\text}

\newcommand{\ep}{\epsilon}

\begin{document}

\title{A class of measures and non-stationary fractals associated to $f$-expansions}
\author{Eugen Mihailescu and Mrinal Kanti Roychowdhury}
\date{}
\maketitle

\begin{abstract}

We construct first a class of Moran fractals in $\mathbb R^d$ with countably many generators and non-stationary contraction rates;  at each step $n$, the contractions depend on $n$-truncated sequences, and are related to asymptotic letter frequencies. In some cases the sets of contractions may be infinite at each step. We show that the Hausdorff dimension of such a fractal is equal to the zero $h$ of a pressure function. We prove that the dimensions of these sets depend \textit{real analytically} on the frequencies.  \  Next, we apply the above construction to obtain non-stationary fractals $E(x; f) \subset \mathbb R^d$ associated to \textit{$f$-expansions} of real numbers $x$, and study the dependence of these fractals on $x$. We consider for instance $\beta$-expansions, the continued fraction expansion and other $f$-expansions.
By employing the Ergodic Theorem for invariant absolutely continuous measures and equilibrium measures,  and using some probabilities for which the digits become independent random variables, we study the function $x \mapsto \text{dim}_\text{H}(E(x; f))$ on the respective set of quasinormal numbers $x \in [0, 1)$.
We investigate also another class of fractals $\tilde E_f(x) \subset \mathbb R^d$ for which, \textit{both} the non-stationary contraction vectors \textit{and} the asymptotic frequencies depend on the $f$-representation of $x$. We obtain then some properties of the digits of $x$, related to $\tilde E_f(x)$ and to equilibrium measures.

\end{abstract}

\



\section{Introduction and outline of results.}

In this paper we investigate a class of \textit{Moran fractals} in $\mathbb R^d$ generated by countably many similarities, which are composed according to certain rules that depend on each iteration step. We study dimension functions for families of non-stationary fractals associated to various $f$-expansions of real numbers $x \in [0, 1)$. In some cases the constructed fractals depend on all the digits of $x$ in $f$-expansion, and we obtain relations between these fractals, the behaviour of the  digits of $x$, and equilibrium measures $\mu_\phi$ associated to $T_f$ and to  potentials $\phi$ on $[0, 1)$. 

First, let $X \subset \mathbb R^d$ be a nonempty compact subset such that $\te{cl(int}(X))=X$. Let $S_n: X \to X,  \ n\geq 1$ be an infinite system of contractive similarities with similarity ratios $c_n$, i.e.
\[|S_n(x)-S_n(y)|=c_n|x-y| \te{ for all }  x, y \in X,\]
where $0<c_n<1$, and $\sup_{n\geq 1}c_n<1$. Denote by $\C K(X)$ the class of nonempty relatively compact subsets of $X$, and the set function $\C S : \C K(X) \to \C K(X)$, $\C S(A)=\UU_{n\geq 1} S_n(A)$.
The sets $S_{i_1}\circ\cdots \circ S_{i_k}(A)$, where $i_k\geq 1$ for any $k\geq 1$, are called the \textit{basic elements} of order $k$.
If   $\sup_{n\geq 1}c_n<1$, then in \cite{M} was obtained a self-similar $\mathcal{S}$-invariant set $E \sci X$ defined by $$E = \mathop{\cup}\limits_{(i_1, i_2, \ldots) \in \mathbb N^\infty} \mathop{\cap}\limits_{k \ge 1} S_{i_1} \circ \ldots \circ S_{i_k}(X)$$

For $s\geq 0$, let $\C H^s(E)$ denote the $s$-dimensional Hausdorff measure, and dim$_{\te{H}}(E)$ denote the Hausdorff dimension of $E$. A set $E$ is called an $s$-set, if $0<\C H^s(E)<\infty$ (\cite{F}, \cite{Ma}). \

In this paper, we first construct a new class of self-similar sets, such that at each step of the construction, the number of basic elements is infinite, the Open Set Condition is satisfied, and the contraction rates at step $n$ vary depending on $n$ and on the asymptotic frequencies with which certain letters appear in the corresponding truncated sequence.  \ 
In our case we use a double-index sequence of similarities $\{S_{i, j}\}_{i, j \ge 1}$ of the compact set $ X \subset \mathbb R^d$,  which satisfy Open Set Condition. Our fractal sets defined in (\ref{def}) are of the form $$E = \mathop{\bigcup}\limits_{(i_1, i_2, \ldots) \in \mathbb N^\mathbb N} \mathop{\bigcap}\limits_{j=1}^\infty  S_{j, i_j} \circ \ldots \circ S_{1, i_1}(X),$$
where the similarities $S_{i, j}$ will be determined by some fixed sequence $\omega = (\omega_1, \omega_2, \ldots)$, such that $\omega$ has letter frequencies given by a stochastic vector $\eta$.

For such an infinitely generated Moran set $E$ with non-stationary contraction rates and stochastic frequency vector $\eta$, we define the \textit{pressure function} $P(t)$, and show that there exists a unique nonnegative number $h$ such that $P(h)=0$.\
Then we show in \textbf{Theorem \ref{theorem}} that
\[\text{dim}_{\te{H}}(E)=h, \ \te{and} \ \C H^h(E) <\infty\]
Unlike for finitely generated Moran sets, even a stationary infinitely generated Moran set $E$ may not be an $h$-set,  except when $h\geq d$. An example in this regard is given in \cite{M}.

Let us emphasize that our Moran sets are different from the infinitely generated sets with stationary contraction rates of Moran (\cite{M}), since in that case the contraction rate at step $n+1$ does not depend on $n$, whereas in our case it depends. Also our Moran sets are different from those of Barreira (see \cite{Ba}), since in that case the sets are modeled by a symbolic system on a finite number of elements. 
In our case the Moran set $E(\omega)$ associated to some sequence $\omega = (\omega_1, \omega_2, \ldots)$ has infinitely many generators, and moreover the contraction rates at step $n$ depend on the truncated sequence $(\omega_1, \ldots, \omega_n)$, so the contraction rates are non-stationary (see \cite{Ba}, \cite{P} for non-stationarity). 

Next, in \textbf{Theorem \ref{real-an}} we prove that the Hausdorff dimensions of the non-stationary Moran fractals obtained above, depend \textbf{real-analytically} jointly on the frequencies, if  the contraction rates are fixed. This allows one to apply the powerful properties of real analytic functions in higher dimensions, such as the determination of the function by its derivatives at a point, the fact that the level sets are real analytic subvarieties,  the Lojaciewicz Vanishing Theorem, etc.

\

Then, our next direction  in Section 4, will be to apply the above construction in order to obtain a \textbf{family} of infinitely generated non-stationary self-similar fractal sets $E(x; f) \subset \mathbb R^d$, \textbf{associated to $f$-expansions} of  real numbers $x \in [0, 1)$.  \ 

The first example will be for the $m$-ary expansion, where $m \ge 2$ is an integer. 
It is known that for Lebesgue-almost all $x \in (0, 1)$ the frequencies of appearance of the digits $0, \ldots, m-1$ in the $m$-ary expansion of $x$, approach  the uniformly distributed probability vector $(\frac 1m, \ldots, \frac 1m)$; \ this follows from the Strong Law of Large Numbers for the independent identically distributed random variables given by the digits, see \cite{D} (or from the Ergodic Theorem for the Lebesgue measure, invariant for $T_m(x) = mx  \ \text{mod} \ 1$, see \cite{Wa}). We say  such numbers $x$ are \textit{normal} for the $m$-ary expansion and denote their set by $\mathcal{N}(m)$. 
However, there exist $x \in [0, 1)$ which are \textit{not normal} but for which the frequencies of digits converge to an arbitrary stochastic vector $(p_0, \ldots, p_{m-1})$; we call such numbers $(p_0, \ldots, p_{m-1})$-quasinormal and denote their set by $$F(p_0, \ldots, p_{m-1}):= \{ x = \mathop{\sum}\limits_{k \ge 1} \frac{d_k(x)}{m^k}, \ \frac{Card\{j, \ 1 \le j \le n, \ d_j(x) = i \}}{n} \mathop{\to}\limits_{n \to \infty} p_i, \  0 \le i \le m-1\}$$ 
We denote by $\mathcal{Q}_m$ the union of all sets $F(p_0, \ldots, p_{m-1})$ over all stochastic vectors $(p_0, \ldots, p_{m-1})$ (i.e vectors which satisfy $p_i > 0, 0 \le i \le m-1$ and $\mathop{\sum}\limits_{0 \le i \le m-1} p_i = 1$). The elements of $\mathcal{Q}_m$ will be called \textit{quasinormal numbers}, and clearly $\mathcal{N}(m) \subset \mathcal{Q}_m$. \ 
We obtain next a functional $E(\cdot; m)$ which assigns to every point $x \in \mathcal{Q}_m$, a Moran set $E(x; m) \subset \mathbb R^d$ with infinitely many generators and non-stationary contraction rates by the above procedure.
This gives a function describing the Hausdorff dimensions of these Moran sets, $H_m: \mathcal{Q}_m \to [0, \infty), \ H_m(x) := \text{dim}_\text{H}(E(x; m)), \ x \in \mathcal{Q}_m$.

Another construction of Moran sets $E(x; \beta) \subset \mathbb R^d$ with infinitely many generators and non-stationary contraction rates, can be done using the digits in the \textbf{beta-expansion}; in this case we have also finitely many possible values for the digits, namely $0, \ldots, [\beta]$. The Lebesgue measure is not $T$-invariant by the map $T(x) = \beta x \ \text{mod} \ 1$ anymore, but we can apply the Ergodic Theorem to the measure $\nu_\beta$ introduced by R\'enyi (\cite{Re}), which is $T$-invariant ergodic and absolutely continuous with respect to Lebesgue measure. Thus we obtain normal numbers for the $\beta$-expansion. \ 
As before we can define also  sets $F(p_0, \ldots, p_{[\beta]})$ and we will compute the Hausdorff dimension of $F(p_0, \ldots, p_{[\beta]})$ in Lemma \ref{dim-beta}.  

For the construction of $E(x; \beta) \subset \mathbb R^d$ we take finitely many infinite positive vectors $\Psi_i, 0 \le i \le [\beta] $ and, at each step $n$, if the $n$-th digit of $x$ in its $\beta$-expansion is $j \in \{0, \ldots, [\beta]\}$ then we use the contraction rates given by $\Psi_j$. \
We can introduce then a dimension function defined on the set $\mathcal{Q}_\beta$ of all \textit{quasinormal numbers} in $[0, 1)$ for the beta-expansion, given by $$H_\beta(x) := \text{dim}_\text{H}(E(x; \beta)), \ x \in \mathcal{Q}_\beta$$

We show in \textbf{Theorem \ref{thm2}} that, in spite of the apparent continuity of the process of construction of the infinitely generated Moran sets $E(x; m)$ and $E(x; \beta)$,  the functions $H_m, H_\beta$ are actually \textbf{highly discontinuous} on the set of quasinormal numbers.  We obtain  that $H_m$ and $H_\beta$  are constant on sets of positive dimension $F(p_0, \ldots, p_{m-1})$, respectively $F(p_0, \ldots, p_{[\beta]})$, and moreover that  given any $x \in \mathcal{Q}_m$, respectively any $x \in \mathcal{Q}_\beta$ and any neighbourhood $U$ of $x$, then the functions $H_m$ and $H_\beta$ take their \textit{entire} range of values already inside $U$.

In Corollary \ref{H-osc} we will give also an estimate of the oscillation of $H_m$ when the frequency vectors vary in some compact set $K \subset \mathbb R^m$, and show that we can have $H_m(x) = H_m(y)$ for two quasinormal numbers $x, y \in \mathcal{Q}_m$, if and only if the frequency vector $\eta(y)$ of $y$ belongs to a \textbf{real-analytic subvariety} $Z_x \subset \mathbb R^m$ which depends on $x$. Similar results hold also for the dimension function $H_\beta(\cdot)$ associated to the $\beta$-expansion.

\

We saw above, that the dimension of $E(x; m)$ is constant when  $ x \in F(p_0, \ldots, p_{m-1})$, provided that the infinite contraction vectors $\Psi_i$ are fixed. However, for a fixed stochastic vector $(p_0, \ldots, p_{m-1})$ and for numbers $x \in F(p_0, \ldots, p_{m-1})$, we introduce in subsection \textbf{\ref{tildeE}, another class} of Moran fractals $\tilde E_m(x)$, whose non-stationary \textbf{contraction vectors $\{\Psi_i(x)\}_{1 \le i \le m}$ depend on $x$}. Such fractals use the whole information about the digits of $x$ in its $m$-expansion  (not only their asymptotic frequencies $p_i$). In this case $\text{dim}_\text{H}(\tilde E_m(x))$ depends on $x$, even when $x$ varies in a set $F(p_0, \ldots, p_{m-1})$ and, in Corollary \ref{cortildeE} we obtain a formula for this dimension. 

In Remark \ref{remtilde} we show that the fractal construction $\tilde E_m(x)$ gives us a way to \textit{distinguish} between numbers $x \in [0, 1)$ which have the same fixed asymptotic frequency vector $(p_0, \ldots, p_{m-1})$. \ 
A similar kind of fractals $\tilde E_\beta(x)$ is constructed using $\beta$-expansions. This may be  done also for the Bolyai-R\'enyi expansion, and other expansions with finitely many digit values as in \ref{f-tilde}.

Then, in \textbf{Theorem \ref{continued}} we investigate non-stationary Moran fractals $E(x; G)$ associated to the \textbf{continued fraction expansion} of real numbers  $x \in [0, 1)$ (see for eg. \cite{DK}). This expansion is induced by the map $T(x) = \frac 1x - [\frac 1x], x \neq 0$, and  $x$ is represented as $x = \frac{1}{d_1(x) + \frac{1}{d_2(x) + \frac{1}{d_3(x) + \ldots}}}$. \ In this case we will need to extend our results to the case of infinitely many contraction rates at each step, i.e. we have an infinite collection of infinite positive vectors $\Psi_i, i \ge 1$ prescribing the non-stationary contraction rates. This is due to the fact that the digits of the continued fraction expansion may take infinitely many values.  \
In this case we shall use some results related to  the invariant absolutely continuous Gauss measure $\mu_G$, and to the dimensions of probability measures that make the digits of the continued fraction expansion to be independent random variables (see for eg. \cite{KP}). \
We study then similarly the dimension function $H_G: \mathcal{Q}_G \to [0, \infty)$ defined for quasinormal numbers $x$ in continued fraction expansion, by
$$H_G(x) := \text{dim}_\text{H}(E(x; G)), \ x \in \mathcal{Q}_G$$

In subsection \textbf{\ref{f-exp}} our results will be applied also to a family of infinitely generated Moran fractals $E(x; f) \subset \mathbb R^d$ constructed with the help of \textbf{other $f$-expansions}, for strictly monotonic functions $f$  (see \cite{A}, \cite{KP}, \cite{Re}, \cite{Wa1}, etc). \
By employing the Perron-Frobenius operator, Walters obtained in \cite{Wa1} a probability measure $\mu_T$ which is absolutely continuous and $T$-invariant, for the map $T(x) = f^{-1}(x) - [f^{-1}(x)], x \in [0, 1)$.
We use also some properties of invariant measures obtained as distributions of random variables associated to $f$-expansions. And, especially we use properties of \textit{equilibrium measures $\mu_\phi$}  associated to $T$  and to piecewise H\"older-continuous potentials $\phi$, in order to control the size of sets of numbers $x \in [0, 1)$ for which we obtain certain dimensions of the fractals $E(x; f)$.

In the end, in subsection \textbf{\ref{f-tilde}} we will study also \textbf{another class} of non-stationary fractals $\tilde E_f(x)$ associated to $\eta$-quasinormal numbers $x$ for the $f$-expansion; we  use stochastic vectors $\eta$ determined by $T$-invariant equilibrium measures $\mu_\phi$ of piecewise H\"older potentials $\phi$.  \
For these fractals in \ref{f-tilde}, the infinite contraction vectors $\Psi_i(x), i \ge 1$ will depend on \textbf{all the digits} of $x$.
This will help to relate some \textbf{properties of the digits} of $x$, like the \textit{speed of convergence} of their relative frequencies towards asymptotic frequencies,  with the behavior of the dimensions of the associated fractals $\tilde E_f(x)$. \ 

 In particular, we obtain results vis-\'a-vis the behavior of $\tilde E_f(x)$ and of the $f$-representation of $x$, with respect to the unique absolutely continuous $T$-invariant probability $\mu_T$.  These results can then be applied to the continued fraction expansion with its Gauss measure $\mu_G$, and its associated family of equilibrium measures on $[0, 1]$.

\section{Construction of the class of Moran sets, and the pressure function.}

Let $\D R^d$ denote the $d$-dimensional Euclidean space.
Let $s\geq 0$, and let $\C U=\set{U_i}$ be a countable collection of sets of $\D R^d$. We define
$\|\C U\|^s : =\sum_{U_i\in \C U} |U_i|^s,$ \ 
where $|A|$ denotes the diameter of a set $A$. Let $E\sci \D R^d$ and $\gd>0$. A countable collection $\C U=\set{U_i}\sci \D R^d$ is called a $\gd$-covering of the set $E$ if $E \sci \UU U_i$ and for each $i$, $0<|U_i|\leq \gd$. Suppose that $E\sci \D R^d$ and $s \geq 0, \delta>0$, then the $s$-dimensional Hausdorff measure of $E$ is defined by:
\[\C H_\gd^s(E) =\inf_{\C U}\left\{\|\C U\|^s : \C U \te{ is a $\gd$-covering of } E\right\}, \ \text{and} \ \C H^s(E)=\lim_{\gd \to 0} \C H_\gd^s(E).\]
 The Hausdorff dimension of $E$, denoted by dim$_{\te{H}}E$, is then defined by
\[\te{dim}_{\te{H}}E=\sup\set{s \geq 0 : \C H^s(E)=\infty} =\inf\set{s \geq 0 : \C H^s(E)=0}.\]

The following proposition is known.

\begin{prop} (\cite{F}, \cite{Ma}) \label{prop22}
Let $E\sci \D R^d$ be a Borel set,  $\mu$ be a finite Borel measure on $\D R^d$ and $0<c<\infty$. If $\mathop{\limsup}\limits_{r\to 0} \mu(B(x, r))/r^s \leq c$ for all $x \in E$, then $\C H^s(E) \geq \mu(E)/c$.
\end{prop}

In general, a family of similarities $\set{S_n}_{n\geq 1}$ is said to satisfy the Open Set Condition  if there exists a nonempty open set $U \sci X$ (in the topology of $X$) such that $S_i(U) \sci U, i \geq 1$ and $S_i(U) \ii S_j(U)=\es$ for all $i, j\geq 1$ with $i\neq j$ (\cite{MU}). 
For self-similar sets of type $E = \mathop{\cup}\limits_{(i_1, i_2, \ldots) \in \mathbb N^\infty} \mathop{\cap}\limits_{k \ge 1} S_{i_1} \circ \ldots \circ S_{i_k}(X)$ with similarities $S_i$ of contraction rate $c_i$, which satisfy the Open Set Condition, it was proved in \cite{MU} and \cite{M} that $\text{dim}_\text{H}(E)$ is given as $s = \inf \{t \ge 0, \ \mathop{\sum}\limits_{i \ge 1} c_i^t \le 1\}$.  \
In \cite{Mi}, \cite{Mi1}, \cite{FM},  relations between fractal dimensions, equilibrium measures on fractal sets and iterated schemes were also studied in certain settings. 

We will now detail our fractal construction:

\subsection{Infinitely generated Moran set with non-stationary contraction rates:} Let $\set{\Phi_k}_{k\geq 1}$  be a sequence of positive vectors in the infinite dimensional space $\mathbb R^{\mathbb N^*}$ such that for all $k\geq 1$,
\[\Phi_k=(c_{k1}, c_{k2}, \cdots ), \ \ c_{kj} \ge 0, k, j \ge 1 \ \text{satisfying} \] \[\sum_{j=1}^{\infty} c_{kj} < \infty,  \te{ and } \sup\set{c_{kj} : j \in \D N}<1\]

Let $D_0$ be the empty set. For $k\geq 1$ write
\[D_{m, k}=\set{(i_m, i_{m+1}, \cdots, i_k) : i_j \in \D N, \, m\leq j\leq k},\]
and $D_k=D_{1, k}$, $D_\infty=\lim_{k\to \infty} D_k$. Define $D:=\uu_{k\geq 0} D_k$. Elements of $D$ are called words. For any $\gs \in D$ if $\gs=(\gs_1, \gs_2, \cdots, \gs_n)  \in D_n$, we write $\gs^-=(\gs_1, \gs_2, \cdots, \gs_{n-1})$ to denote the word obtained by deleting the last letter of $\gs$,  $|\gs|=n$ to denote the length of $\gs$, and $\gs|_k:=(\gs_1, \gs_2, \cdots, \gs_k)$, $k\leq n$, to denote the truncation of $\gs$ to the length $k$.  If  $\gs=(\gs_1, \gs_2, \cdots, \gs_k) \in D_k$ and $\gt=(\gt_1, \gt_2, \cdots, \gt_m) \in D_{k+1, m}$, then we write $\gs\gt=\gs\ast \gt=(\gs_1,  \cdots, \gs_k, \gt_1,  \cdots, \gt_m)$ to denote the juxtaposition of $\gs, \gt \in D$.   For $\gs \in D$ and $\gt\in D\uu D_\infty$ we say $\gt$ is an extension of $\gs$, written as $\gs\prec \gt$, if $\gt|_{|\gs|}=\gs$.

Let $J $ be a nonempty compact subset of $\mathbb R^d$ such that $J=\te{cl(int} J)$, where $\te{int}(A)$ denotes the interior of a set $A$, and  let $\F F=\set{J_\gs : \gs \in D}$ be an infinite collection of nonempty subsets of $\D R^d$. We say that $\F F$ fulfills the conditions for the \textbf{infinite Moran structure with non-stationary contraction rates} (\textbf{IMSNC}), if it satisfies the following:

$1)$ $J_\es=J$.

$2)$ For all $\gs \in D$, $J_\gs$ is similar to $J$, i.e there is a similarity $S_\gs: \D R^d\to \D R^d$ with $J_\gs=S_\gs(J)$.

$3)$ For any $k\geq 0$ and $\gs \in D_k$, $J_{\gs\ast1}, J_{\gs\ast 2}, \cdots$ are subsets of $J_\gs$, and $\te{int}(J_{\gs\ast i})\ii \te{int}(J_{\gs\ast j})=\es$ for $i, j\geq 1$ and $i\neq j$.

$4)$ For any $k\geq 1$ and $\gs \in D_{k-1}$, $j\geq 1$, $\frac{|J_{\gs\ast j}|}{|J_\gs|}=c_{kj}$.

\begin{remark}
Since $D_k$ is infinite for $k\geq 1$, the set $\uu_{\gs\in D_k} J_\gs$ may not be closed. 
\end{remark}
The nonempty Borel set $E:=E(\F F)$ defined by
 $$E = \UU\limits_{\sigma \in D_\infty}\II\limits_{1 \le i < \infty} J_{\sigma|_i}$$
 is called the Moran set (or Moran fractal) associated with the collection $\F F$. As $\F F$ satisfies the infinite Moran structure conditions, we call $E$ an \tit{infinitely generated Moran set with non-stationary contraction coefficients}, or simply \textit{an infinitely generated non-stationary Moran fractal}. \ Let us mention that a similar kind of non-stationary construction as above was done in the finitely generated case in \cite{HRW}, \cite{We}. 

Let now $\F F_k=\set{J_\gs : \gs\in D_k}$, and $\F F=\uu_{k\geq 0} \F F_k$. The elements of $\F F_k$ are called the basic sets of order $k$, and the elements of $\F F$ are called the basic elements of the Moran set $E$. 
 \begin{remark}
 If $\lim_{k\to \infty} \sup_{\gs \in D_k} |J_\gs|>0$, then $E$ contains interior points. Then the measure and dimension properties will be trivial. We assume therefore
 $ \lim_{k\to \infty} \sup_{\gs \in D_k} |J_\gs|=0$.
 \end{remark}
 By the definition of the Moran set $E$ associated with $\F F$, we see that $\F F$ is a net of $E$, i.e., for any $x \in E$ and any $\ep>0$, there is $G \in \F F$ such that $x \in G$ and $|G|<\ep$. Two basic elements are said to be disjoint if their interiors are disjoint. 

Suppose that the set $J$ and the collection $\set{\Phi_k}_{k\geq 1}$ are given; we denote by $\C M(J, \D N, \set {\Phi_k})$ the class of Moran sets satisfying the IMSNC. We call $\C M(J, \D N, \set {\Phi_k})$ the Moran class associated with the triplet $(J, \D N, \set {\Phi_k})$; this class is obtained by considering all the possible similitudes $S_\sigma, \sigma \in D,$ which satisfy the conditions IMSNC above.

Next, without loss of generality we will assume that the diameter of the set $J$ is one. \
Let now $A:=\set{a_1, a_2, \cdots, a_m}$ a finite set of positive integers, and define the set of sequences
\[A^{\D N}:=\set{(t_j)_{j=1}^\infty \in D_\infty : t_j\in A}\]
Consider $\go=(s_1, s_2,  \cdots) \in A^{\D N}$, let an integer $k \ge 1$ and $\go|_k=(s_1, s_2, \cdots, s_k)$ and  define the number of times that $a_i$ appears in the truncated word $\omega|_k$ by  \ $$\|\go|_k\|_{a_i}:= Card \{j, \ s_j=a_i, \ 1 \le j \le k\}$$ 

In addition assume that the sequence $\go\in A^{\D N}$ satisfies: \ $\lim_{k\to \infty} \frac{\|\go|_k\|_{a_i}}{k} =\gh_i>0,$ \ for every $a_i\in A, 1 \le i \le m$;  in this case we say that  $\go$ has the \textit{frequency vector} $\gh=(\gh_1, \gh_2, \cdots, \gh_m)$, where $\mathop{\sum}\limits_{i = 1}^m \eta_i = 1$ (i.e $\eta$ is a \textit{stochastic vector}). \
 Let $T$ be the left shift on $A^{\D N}$, i.e., $T(\go)=(s_2, s_3, \cdots)$ where $\go=(s_1, s_2, \cdots) \in A^{\D N}$. Note that $\sum_{i=1}^m\|\go|_k\|_{a_i}=k$. For a stochastic vector $\gh=(\gh_1, \gh_2, \cdots, \gh_m) \in (0, 1)^m$, define the set of sequences:
\[A_\gh^{\D N}:=\set {\rho=(\rho_j)_{j\geq 1} : \rho_j\in A, \, \lim_{k\to \infty} \frac{\|\rho|_k\|_{a_i}}{k} =\gh_i, \, 1\leq i\leq m}\]

Let us now fix $\Psi_i=(\gamma_{i1}, \gamma_{i2}, \cdots), 1 \le i \le m$  infinite positive vectors satisfying the condition:
\begin{equation}\label{gamma}
\sum_{j=1}^{\infty}\gamma_{ij} < \infty, 1 \le i \le m, \ \text{and} \ 
\sup\set{\gamma_{ij},   1 \le i \le m, j \ge 1}<1
\end{equation}
  If $\go =(s_1, s_2, \cdots)\in A_{\gh}^{\D N}$, then for $k \ge 1$ arbitrary, if it happens that  the element $s_k$ of $\omega$ is equal to  some $a_i \in A$, then we define the infinite positive vector $$\Phi_k = (c_{k1}, c_{k2}, \ldots):= \Psi_i = (\gamma_{i1}, \gamma_{i2}, \cdots )$$

In this way, we obtain a class of Moran sets associated with $\go \in A_\gh^{\D N}$, by taking all possible similitudes that satisfy IMSNC.  Denote such a generic Moran set by 
\begin{equation}\label{def}
E:=E(\go),
\end{equation}
 and call it a \textbf{Moran set with infinitely many generators and non-stationary rates} associated to $\{\Psi_i\}_{1 \le i \le m}$, to the stochastic vector $\eta$ and to the sequence $\omega \in A^\mathbb N_\eta$.  This class of fractals, the measures supported on them and the pressure functionals, together with relations to $f$-expansions and associated random variables, will be the object of study in the sequel.
Later in the paper we extend the construction to the case of countably many vectors $\Psi_i$. 

Let now
$C=\sup\set{\gamma_{ij} : j\in \D N, \, 1\leq i\leq m}$. By our assumption,  $0<C<1$. Let us define
\[\gq_i=\inf \set{t >0 : \sum_{j=1}^\infty \gamma_{ij}^t<\infty }, 1 \le i \le m, \te{ and } \gq=\max\set{\gq_i : 1\leq i\leq m}\]
Note that $\gq_i\geq 0, \ 1\leq i\leq m$, and thus only one of the following two disjoint cases can occur:

$(C1)$ $\sum_{j=1}^\infty \gamma_{ij}^\gq =\infty$ for some $i$;

$(C2)$ $\sum_{j=1}^\infty \gamma_{ij}^\gq <\infty$ for all $i$.

Let us assume that in our paper $(C1)$ happens, and then $\sum_{j=1}^\infty \gamma_{ij}^t<\infty$ for all $t\in (\gq, +\infty)$ and for all $1\leq i\leq m$. If $(C2)$ happens, then $\sum_{j=1}^\infty \gamma_{ij}^t<\infty$ for all $t\in [\gq, +\infty)$ and $1\leq i\leq m$.

\subsection{Pressure function.} Recall that we fixed an element $\omega \in A_\eta^{\mathbb N}$, and a sequence of infinite vectors $\set{\Phi_k}_{k \ge 1}$, with  $\Phi_k = (c_{k1}, c_{k2}, \ldots), k \ge 1$ defined in terms of $\omega$ and of $\Psi_1, \ldots, \Psi_m$. For an arbitrary $\gs=(\gs_1, \gs_2, \cdots, \gs_k) \in D_k$, let us define
\begin{align*}
c_\gs=\left\{\begin{array} {ll} c_{1\gs_1} c_{2\gs_2} \cdots c_{k\gs_k} & \te{ if } k\geq 1\\
1 &\te{ if } k=0.
\end{array} \right.
\end{align*}
Let $t\in (\gq, +\infty)$ be given, and define the sequence
\[c_k(t)= \frac 1 k \log \sum_{\gs\in D_k}c_\gs^t=\frac 1 k \log \prod_{i=1}^m\left(\sum_{j=1}^{\infty} \gamma_{ij}^t\right)^{\|\go|_k\|_{a_i}}= \sum_{i=1}^m \frac {\|\go|_k\|_{a_i}}{k}\log \sum_{j=1}^{\infty}  \gamma_{ij}^t,\]
for $k\in \D N$.
Let us denote by $L(t)=\sum_{i=1}^m\left |\log \sum_{j=1}^{\infty} \gamma_{ij}^t\right |$. Note that $0<\sum_{j=1}^{\infty} \gamma_{ij}^t <\infty$ for all $t \in (\gq, +\infty)$ and for all $1\leq i \leq m$. Thus $L(t)$ is a positive number for any $t\in (\gq, +\infty)$. \

We claim that $\set{c_k(t)}_{k=1}^\infty$ is a Cauchy sequence.
For all $1\leq i\leq m$,  the sequence $\left\{\frac{\|\go|_k\|_{a_i}}{k}\right\}_{k=1}^\infty$ is convergent, and so for any $\varepsilon>0$ there exists a positive integer $N$ such that for all $n, p\geq N$ and for all $1\leq i\leq m$,
$\left|\frac{\|\go|_n\|_{a_i}}{n}-\frac{\|\go|_p\|_{a_i}}{p}\right|<\frac{\varepsilon}{L(t)}$.
Using this fact, and noting that
\begin{align*}
c_n(t)-c_p(t) =\sum_{i=1}^m\left (\frac{\|\go|_n\|_{a_i}}{n}- \frac{\|\go|_p\|_{a_i}}{p}\right) \log \sum_{j=1}^{\infty} \gamma_{ij}^t,
\end{align*}
for all $n, p\geq N$, we then obtain:
\[\left|c_n(t)-c_p(t)\right| \leq \sum_{i=1}^m\left|\frac{\|\go|_n\|_{a_i}}{n}- \frac{\|\go|_p\|_{a_i}}{p}\right| \left|\log \sum_{j=1}^{\infty} \gamma_{ij}^t\right|< \frac{\varepsilon}{L(t)} \sum_{i=1}^m \left|\log \sum_{j=1}^{\infty} \gamma_{ij}^t\right| <\varepsilon,\]
and thus the claim is true. Notice also that from above, the sequence $\set{c_k(t)}_{k=1}^\infty$ is bounded; so $\set{c_k(t)}_{k=1}^\infty$ is a convergent sequence.
Hence we can define a real-valued function $P(t)$ as:
\begin{equation} \label{eq1} P(t)=\lim_{k\to \infty}c_k(t)=\lim_{k\to\infty}\sum_{i=1}^{m} \frac {\|\go|_k\|_{a_i}}{k}\log \sum_{j=1}^{\infty}  \gamma_{ij}^t=\sum_{i=1}^m\gh_i\log \sum_{j=1}^{\infty} \gamma_{ij}^t,\end{equation}
where $t \in (\gq, +\infty)$. \ 
The above function $P(t)$ will be called the \tit{pressure function} corresponding to the Moran set $E:=E(\go)$, by analogy to the usual topological pressure for continuous functions \cite{Wa}.
The pressure function associated to  $E(\omega)$ depends only on the contraction rates $c_{kj}, k, j \ge 1$, which in turn depend on $\omega$ and $\Psi_1, \ldots, \Psi_m$. \

We will assume also that there exists some $u\in (\gq, +\infty)$ with $0<P(u)<\infty$.
The following two lemmas yield some properties of the function $P(t)$, similar to the classical notion of pressure for a dynamical system (\cite{Bo}).

\begin{lemma} \label{lemma1}
The function $P(t)$ is strictly decreasing, convex and continuous on $(\gq, +\infty)$.

\end{lemma}
\begin{proof}  To prove that $P(t)$ is strictly decreasing let $\gd>0$. Then,
$$
P(t+\gd)=\sum_{i=1}^m \gh_i \log\sum_{j=1}^{\infty} \gamma_{ij}^{t+\gd}\leq \sum_{i=1}^m \gh_i \log\sum_{j=1}^{\infty} \gamma_{ij}^{t} C^\gd
=P(t) +\gd\log C<P(t),
$$
as $\gd \log C<0$ and so $P(t)$ is strictly decreasing. For $t_1, t_2 \in (\gq, +\infty)$ and  $a_1, a_2 > 0$ with $a_1+a_2=1$, using H\"older's inequality, we have
\begin{align*}
P( a_1 t_1+a_2t_2)&= \sum_{i=1}^m \gh_i \log\sum_{j=1}^{\infty}  \gamma_{ij}^{a_1 t_1+a_2 t_2}= \sum_{i=1}^m \gh_i \log\sum_{j=1}^{\infty}\gamma_{ij}^{a_1t_1}
 \gamma_{ij}^{a_2t_2}\\
&\leq \sum_{i=1}^m \gh_i \log \Big(\sum_{j=1}^\infty  \gamma_{ij}^{t_1}\Big)^{a_1}\Big( \sum_{j=1}^\infty   \gamma_{ij}^{t_2}\Big)^{a_2}=a_1P(t_1)+a_2 P(t_2),
\end{align*}
i.e $P(t)$ is convex on $(\gq, +\infty)$. Combined with the fact proved above that for small $\delta>0$,  $P(t+\delta) \le P(t) + \delta \log C$,  we obtain then that $P(t)$ is continuous on $(\theta, \infty)$.

\end{proof}

\begin{lemma} \label{lemma11}
There exists a unique $h\in \D R$ such that $P(h)=0$. In addition $h\in (\gq, +\infty)$.
\end{lemma}

\begin{proof}
By Lemma~\ref{lemma1}, the function $P(t)$ is strictly decreasing and continuous on $(\gq, +\infty)$. Since $0<P(u)<\infty$ for some $u\in (\gq, +\infty)$, in order to conclude the proof it therefore suffices to show that $\lim_{t \to +\infty} P(t)=-\infty$. For $t> u$,
$$
P(t)=\sum_{i=1}^m \gh_i \log\sum_{j=1}^{\infty}\gamma_{ij}^t=\sum_{i=1}^m \gh_i \log\sum_{j=1}^{\infty}\gamma_{ij}^u \gamma_{ij}^{t-u}
\leq \sum_{i=1}^m \gh_i \log\sum_{j=1}^{\infty}\gamma_{ij}^u C^{t-u}=P(u)+(t-u) \log C
$$
Since $C<1$, it follows that $\lim_{t\to+\infty}P(t)=-\infty$, and thus the lemma follows.

\end{proof}

\section{Hausdorff and probability measures on fractals of type $E(\omega)$. Real-analytic dependence of dimensions on frequencies.}

In this section we show that the Hausdorff dimension of the infinitely generated non-stationary Moran set $E(\omega) \subset \mathbb R^d$ obtained above, is equal to the unique zero of the pressure functional. Similar results, but for different types of fractals, have been obtained in a variety of settings (see for example \cite{Ba}, \cite{M}, \cite{P}, etc.)  We will study Hausdorff measures on $E(\omega)$ by constructing certain probability measures supported on it. 
In our case, the relationship between the unique zero $h$ of the pressure function $P(t)$, and the Hausdorff dimension $\te{dim}_{\te{H}}(E)$ is given in Theorem \ref{theorem}.

Then, in Theorem \ref{real-an} we will prove that the Hausdorff dimension of the Moran fractal $E(\omega)$ with $\omega \in A^\mathbb N_{\eta}$, \textbf{depends real analytically on the frequencies} $\eta = (\eta_1, \ldots, \eta_m)$, if the contraction rates are fixed. \ 
Let us give now the dimension formula:

\begin{theorem}\label{theorem}
Let $E:=E(\go) \in \C M(J, \D N, \set{\Phi_k})$ for $\go\in A_\gh^{\D N}$, and consider $h\in (\gq, +\infty)$ to be the unique number which satisfies the equation $P(h)=0$. Then
 \[\te{dim}_{\te{H}}(E)=h, \te{ and } \C H^h(E)<\infty\]
Moreover if $h = d$, then it follows that $\C H^d(E)>0$.
\end{theorem}

The following proposition plays an important role in the proof of the above theorem.

\begin{prop} \label{prop1}
Let $h \in (\gq, +\infty)$ be determined by $P(h)=0$. Let $s_\ast$ and $s^\ast$ be any two arbitrary real numbers with $\gq <s_\ast<h<s^\ast$. Then there exists a positive integer $N_0$ independent of $s_*, s^*$ such that for every integer $n>N_0$ we have:
\[\sum_{\gs \in D_n}c_\gs^{s^\ast} <1<\sum_{\gs \in D_n}c_\gs^{s_\ast}\]
\end{prop}

\begin{proof}

Let $s_\ast<h$. As the pressure function $P(t)$ is strictly decreasing, $P(s_\ast)>P(h)=0$. Then for any given positive integer $n$,
$$
0<P(s_\ast) =\mathop{\sum}\limits_{i=1}^m \eta_i \log \mathop{\sum}\limits_{j=1}^\infty \gamma_{ij}^{s_*} = \lim_{n \to \infty}\sum_{i=1}^m  \frac{\|\go|_{n}\|_{a_i}}{n} \log \sum_{j=1}^{\infty} \gamma_{ij}^{s_\ast}
$$

By assumption $\lim_{n\to \infty} \frac{\|\go|_n\|_{a_i}}{n} =\gh_i>0$, so there exists an integer $N_0$ independent of $s_*, s^*$ such that for all $n> N_0$,  $\frac{\|\omega|_n\|_{a_i}}{n} \ge \eta_i/2 >0$, for all $1\leq i\leq m$. Then for all $n >N_0$ we obtain:
\[0<\sum_{i=1}^m \|\go|_{n}\|_{a_i} \log \sum_{j=1}^{\infty} \gamma_{ij}^{s_\ast}=\log \sum_{\gs \in D_n} c_{\gs}^{s_\ast}, \te{ i.e., } 1<\sum_{\gs \in D_n} c_{\gs}^{s_\ast}. \]

To prove the other inequality, let $h<s^\ast$. Then $0=P(h)>P(s^\ast)$, and we obtain: 
$$
0>\sum_{i=1}^m \lim_{n \to \infty} \frac{\|\go|_{n}\|_{a_i}}{n} \log \sum_{j=1}^{\infty} \gamma_{ij}^{s^\ast}
$$

But for $1 \le i \le m$ we know that  $\lim_{n\to \infty} \frac{\|\go|_n\|_{a_i}}{n} =\gh_i>0$. Thus there exists an integer $N_0$ (assumed the same as before without loss of generality) such that, for any $n>N_0$ we have   $\frac{3\eta_i}{2} > \frac{\|\go|_{n}\|_{a_i}}{n} > \frac{\eta_i}{4}>0, \  \forall 1\leq i\leq m$. So from the above inequalities, it follows that \
 $$0>\sum_{i=1}^m \frac{\|\go|_{n}\|_{a_i}}{n} \cdot \log \sum_{j=1}^{\infty} \gamma_{ij}^{s^\ast},$$ \ 
which implies  that $\sum_{\gs \in D_n} c_{\gs}^{s^\ast}<1$; thus the proof of the proposition is complete.
\end{proof}

\begin{cor} \label{cor1}
In the setting of Proposition \ref{prop1} it follows that, for any $n > N_0$ we have \
$\sum_{\gs \in D_n}c_\gs^{h}=1$.
\end{cor}

\begin{proof} Let us consider a fixed $n$ with $n > N_0$, and let us write $f_n(t)=1-\sum_{\gs \in D_n}c_\gs^{t}$, where $t\in (\gq, +\infty)$. Then $f_n(t)$ is a real-valued continuous function on $(\gq, +\infty)$.  Now since the inequalities in Proposition \ref{prop1} hold for our fixed $n$ and for any numbers $s_*, s^*$ with $s_* < h < s^*$, it follows that $1 \le \mathop{\sum}\limits_{\sigma \in D_n} c_\sigma^h \le 1$. In this fashion we thus obtain the required equality.
\end{proof}

We will mostly need the next Proposition for the case of finitely generated non-stationary Moran fractals, but however we prove it now in more generality:

\begin{prop} \label{prop32}
Let $h\in (\gq, +\infty)$ be determined by $P(h)=0$. Then there exists a Borel probability measure $\mu_h$ supported on $\bar E$ such that for any $k\geq 1$ and $\gs_0\in D_k$,
\[ \mu_h(J_{\gs_0}) =\frac{c_{\gs_0}^h}{\sum_{\gs \in D_k} c_\gs^h}\]
In particular if $k$ is large enough, then 
$\mu_h(J_{\gs})=c_{\gs}^h=|J_\gs|^h$, for  $\sigma \in D_k$.
\end{prop}

\begin{proof}
Consider a sequence of probabilities $(\mu_n)_n$ supported by $E$ such that for  any $\gs_0 \in D_n$,
\begin{equation} \label{eq123} \mu_n(J_{\gs_0})= \mu_n(\text{int}(J_{\sigma_0})) = \frac{c_{\gs_0}^h}{\sum_{\gs\in D_n} c_\gs^h}
\end{equation}
More precisely, we can construct $\mu_n$ as follows: \ Distribute the unit mass among the rank-$n$ basic elements according to \eqref{eq123}. Inductively, suppose then that we have already distributed the mass of proportion $\gm_n (J_\gs)$ to basic elements $J_\rho$ where $\rho\in D_p$ for some $p\geq n$. Then redistribute the mass concentrated on $J_\rho$ to each of its subsets $J_{\rho\ast j}$, with proportions $\frac{c_{\rho\ast j}^h}{\sum_{i=1}^{\infty} c_{\rho\ast i}^h}$,  where $j \in \D N$. Thus:
$\gm_n(J_{\rho\ast j})=\frac{c_{\rho\ast j}^h}{\sum_{i=1}^{\infty} c_{\rho\ast i}^h}\mu_n(J_\rho),$
where  $j\in \D N$. Repeating the above procedure, we then get the measure $\mu_n$.\ 
Fix now some $ n\geq 1$, and take $k\leq n$ and $\gs_0 \in D_k$. We want to check the compatibility of the definition of $\mu_n$ on an arbitrary $J_{\sigma_0}$. From the disjointness of the sets $J_\sigma$, we obtain \ 
$\gm_n(J_{\gs_0})=\sum_{\gt \in D_{k+1, n}}\mu_n(J_{\gs_0\ast\gt}).$ \
Then from \eqref{eq123}, we have
\begin{equation} \label{eq124} \sum_{\gs\in D_n} c_\gs^h\mu_n(J_{\gs_0})=\sum_{\gt \in D_{k+1, n}} c_{\gs_0 \ast\gt}^h \end{equation}

For any $\gs\in D_k$ and $\gt \in D_{k+1, n}$ we have:
$\frac{|J_{\gs \ast\gt}|}{|J_{\gs}|} = \frac{|J_{\gs_0\ast\gt}|}{|J_{\gs_0}|} \te{ and so, } c_{\gs} c_{\gs_0\ast\gt}=c_{\gs_0} c_{\gs\ast\gt}$.
Thus by \eqref{eq124}, we have
$c_{\gs}^h \sum_{\gs\in D_n} c_\gs^h\mu_n(J_{\gs_0})= c_{\gs_0}^h \sum_{\gt \in D_{k+1, n}} c_{\gs\ast\gt}^h$, and therefore
\[\sum_{\gs \in D_k} c_{\gs}^h \sum_{\gs\in D_n} c_\gs^h\mu_n(J_{\gs_0})= c_{\gs_0}^h \sum_{\gs\in D_k} \sum_{\gt \in D_{k+1, n}} c_{\gs\ast\gt}^h=c_{\gs_0}^h \sum_{\gs \in D_n} c_\gs^h, \ \text{hence},\]
\begin{equation} \label{eq125} \mu_n(J_{\gs_0})=\frac{c_{\gs_0}^h}{\sum_{\gs \in D_k} c_\gs^h}\end{equation}
In the above way, we construct a sequence of probability measures $\set{\mu_n}_{n\geq 1}$ which are supported on $E$, and satisfy \eqref{eq125} for any $k\leq n$ and $\gs_0 \in D_k$. \
Now since all the measures $\mu_n$ are probability measures,  we are able to extract a subsequence $\set{\mu_{m_n}}_{n=1}^\infty$ converging weakly to a limit measure $\mu_h$. To verify that $\mu_h$ fulfills the desired requirements we use Theorem~1.24 in \cite{Ma}. Fix some $k\geq 1$ and $\gs_0 \in D_k$. Then by the properties of weak convergence, $\limsup_{n\to\infty} \mu_{m_n}(J_{\gs_0}) \leq \mu_h(J_{\gs_0})$. Combining with \eqref{eq125} we get $\frac{c_{\gs_0}^h}{\sum_{\gs \in D_k} c_\gs^h}\leq \mu_h(J_{\gs_0})$. \  In fact this means that for $k$ large enough, we obtain from Corollary \ref{cor1} the inequality 
\begin{equation} \label{ineq1}
c_{\sigma_0}^h \le \mu_h(J_{\gs_0})
\end{equation} 
Recall that  from definition of $\mu_n$ above, $\mu_{m_n}(\te{int}(J_{\gs_0}))= \mu_{m_n}(J_{\gs_0})$ for all $n\geq 1$. 
Since we work with similarities,  the sets $J_\sigma, \sigma \in D$ are conformal images of $J$; they have also arbitrarily small diameters and thus form a Vitali cover of any fixed set $J_{\sigma_0}$. 
From the properties of weak limits of measures on open sets (\cite{Ma}) we have, for any $\sigma \in D$, that
$\liminf_{n\to\infty} \mu_{m_n}(\te{int}(J_{\gs}))\geq \mu_h (\te{int}(J_{\gs}))$.
Therefore if $k$ is large enough (i.e $k > N_0$) and $\sigma \in D_k$, we obtain from Corollary \ref{cor1}:
$$\mu_h(\text{int}(J_{\gs})) \leq \liminf_{n\to\infty} \mu_{m_n}(\text{int}(J_{\gs}))\leq c_{\gs}^h$$

By applying the last inequality for balls of arbitrarily small diameters centered at  arbitrary points $x \in J_{\sigma_0}$, we obtain from \cite{F} or \cite{Ma} that $\mu_h(J_{\sigma_0}) \le c_{\sigma_0}^h$ for any $\sigma_0 \in D_k, k > N_0$. Thus combining with \eqref{ineq1} we obtain that, for  any $k > N_0$ and any $\sigma_0 \in D_k$, 
$\mu_h(J_{\gs_0}) = c_{\gs_0}^h$. 

\end{proof}

Now, from the collection $\F F$ that satisfies the infinite Moran structure conditions,  let us select a partial sequence $\set{\F F_M}_{M\geq 2}$ in the following fashion: Let $M\geq 2$ be a positive integer. Write
\begin{align*}
\Phi_{k, M}&=(c_{k1}, c_{k2}, \cdots, c_{kM}) \te{ for } k\geq 1,\\
D_{\ell, k, M}&=\set{(i_{\ell}, i_{\ell+1}, \cdots, i_k) : 1\leq i_j\leq M, \, \ell \leq j\leq k},
\end{align*}
$D_{k, M}=D_{1, k, M}$, $D_{\infty, M}=\lim_{k\to\infty} D_{k, M}$ and $D_M=\uu_{k\geq 0} D_{k, M}$ where $D_{0, M}=\es$. All other basic notations and definitions are analogous to those of infinite Moran structure as defined in the previous section. Note that $D_M\sci D$, and let $\F F_{M}=\set{J_\gs: \gs \in D_M}$. Clearly $\F F_M\sci \F F$ and $\F F_M$ satisfies all the conditions needed for the finite Moran structure. For the details about such a finite Moran structure one could see \cite{Ro}, \cite{We}. 

If $M$ and $\set{\Phi_{k, M}}_{k\geq 1}$ are given, we denote the class of Moran sets as above, by $\C M(J, M, \set{\Phi_{k, M}})$. Let $A=\set{a_1, a_2, \cdots, a_m}$ and $A_\gh^{\D N}$ remain the same as before.
For the same $\go =(s_1, s_2, \cdots)\in A_{\gh}^{\D N}$ that was chosen before, and $E\in \C M(J, M, \set {\Phi_{k, M}})$,  if $s_k=a_i$, $k\geq 1$, then take  $\Phi_{k, M}:=\Psi_{i, M}=(\gamma_{i1}, \gamma_{i2}, \cdots, \gamma_{i M})$. Then we obtain a class of Moran sets associated with $\go \in A_\gh^{\D N}$. Let us denote such a Moran set in this case by $E_M:=E_M(\go)$.

Let now $P_M(t)$ be the pressure function for the finite Moran structure, defined for $t\in \D R$ by:
\[P_M(t)=\lim_{k\to\infty}\sum_{i=1}^{m} \frac {\|\go|_k\|_{a_i}}{k}\log \sum_{j=1}^{M}  c_{ij}^t=\sum_{i=1}^{m}\gh_i\log \sum_{j=1}^{M}  c_{ij}^t\]
It is easy to see that the function $P_M(t)$ is strictly decreasing, convex and continuous for $t \in \D R$ (see \cite{Ro}).
For finite Moran structures, we have the following:

\begin{lemma} \label{lemma1222} There exists a unique real number $h_M\in (0, +\infty)$ such that $P_M(h_M)=0$. For such an $h_M\in (0, +\infty)$, we have
\[0<\C H^{h_M}(E_M)<\infty \te{ and thus }  \te{dim}_{\te{H}}(E_M)=h_M\]
\end{lemma}

\begin{proof}
For the finite structure $\C M(J, M, \set{\Phi_{k, M}})$ and a corresponding attractor $E_M$, we can prove immediately analogues of Propositions \ref{prop1} and \ref{prop32}. Then by using the mass distribution principle (see \cite{F}) and the cover of $E_M$ with sets of type $J_\sigma, \sigma \in D_{n, M}$, we obtain the conclusion.
\end{proof}

Now let us prove the following result about the behavior of $h_M$, when $M$ increases:

\begin{lemma} \label{lemma12} Let $h \in (\gq, +\infty)$ be as in Lemma~\ref{lemma11}, and $h_M$ be as in Lemma~\ref{lemma1222}. Then $h_M\uparrow h$, i.e., $h_M\leq h_{M+1}$ for all $M\geq 2$ and $h_M \to h$ as $M\to \infty$.
\end{lemma}
\begin{proof} For $M\geq 2$, let us first prove $h_M\leq h_{M+1}$. If not, let $h_M>h_{M+1}$, write $\ep=\frac{h_M-h_{M+1}}{2}$, which implies
$h_M-\ep=h_{M+1}+\ep$. Since $h_M>h_M-\ep$, we have
\[0=P_M(h_M)<P_M(h_M-\ep)=P_M(h_{M+1}+\ep)=\lim_{k\to \infty} \frac 1 k \log \sum_{\gs\in D_{k, M}}c_\gs^{h_{M+1}+\ep}, \ \text{hence}\]
\[0\leq \lim_{k\to \infty} \frac 1 k \log \sum_{\gs\in D_{k, M+1}}c_\gs^{h_{M+1}+\ep}\leq \lim_{k\to \infty} \frac 1 k \log \sum_{\gs\in D_{k, M+1}}c_\gs^{h_{M+1}} C^{k\ep}=P_{M+1}(h_{M+1})+\ep\log C,\]
hence $0\leq \ep\log C<0$, a contradiction. Hence, $h_M\leq h_{M+1}$. Let us now prove that $h_M \to h$ as $M\to \infty$.  It is enough to prove that
\[h\leq \liminf h_M\leq \limsup h_M\leq h.\]
Assume first that $\liminf h_M<h$. Then for some $\gd>0$, there exists a subsequence $(h_{M_k})_{k\geq 1}$ of the sequence $(h_M)_{M\geq 2}$ such that
$h_{M_k} \leq h-\gd <h.$
Since for all $M\geq 2$, $P_M(t)$ is strictly decreasing for all $t \in \D R$, we have
$0=P_{M_k}(h_{M_k})\geq P_{M_k}(h-\gd)$.
Take $k\to \infty$; then $0\geq P(h-\gd) >P(h)=0$, which is a contradiction. Thus $h\leq \liminf h_M$. Similarly we can prove  $\limsup h_M\leq h$, and thus the proof is obtained.
\end{proof}

\begin{prop} \label{prop120} Let $h \in(\gq, +\infty)$ be unique for which $P(h)=0$. Then $\C H^h(E)<\infty$ and dim$_{\te{H}}(E)=h$.
\end{prop}

\begin{proof}
 For any $n\geq 1$, the set $\set{J_\gs : \gs \in D_n}$ is a covering of the set $E$, and so by Corollary~\ref{cor1},
\[\C H^h(E) \leq \liminf_{n\to \infty} \sum_{\gs \in D_n}|J_\gs|^h=\liminf_{n\to \infty} \sum_{\gs \in D_n}c_\gs^h=1<\infty,\]
which yields that dim$_{\te{H}}(E)\leq h$. For any $M\geq 2$, let $E_M$ be a Moran set generated by $\F F_M$ such that $E_M\sci E$, and so by monotonicity of Hausdorff dimension and Lemma~\ref{lemma1222}, we get $\text{dim}_{\te{H}}(E)\geq \te{dim}_{\te {H}}(E_M)=h_M.$ Take then $M\to \infty$ and use Lemma~\ref{lemma12}, and thus we obtain dim$_{\te{H}}(E)\geq  h$.
\end{proof}

From Proposition \ref{prop120} it follows that dim$_{\te{H}}(E) = h$, thus $h \le d$ since $E \subset \mathbb R^d$. \
We now study what happens if the Hausdorff dimension of $E$ is maximal, i.e equal to $d$.

\begin{prop} \label{prop121}   Let $h \in(\gq, +\infty)$ be the unique number for which $P(h)=0$, and assume that $h = d$, where $d$ is the dimension of the underlying space. Then $\C H^d(E)>0$.
\end{prop}
\begin{proof}
Let $x \in J$ and choose $0<r<1$. Let us now consider the ball $B=B(x, r)$. Let
\[\C C=\set{J_\gs \in \F F : J_\gs \II B \neq \es \te{ and } |J_\gs|<r\leq |J_{\gs^-}|}, \
\text{and for} \ k\geq 1,\] \begin{equation} \label{eq234} \C C_k=\set{J_\gs \in \C C : 2^{-k}r\leq   |J_\gs|< 2^{1-k} r}\end{equation}
Then for all $k\geq 1$, $J_\gs$ belonging to $\C C_k$ is contained in a ball $B'$ concentric with $B$  and of radius $2r$. Moreover, elements of $\C C_k$ are disjoint and $\C C=\uu_{k=1}^\infty \C C_k$. Since int$(J)$ is nonempty, the set $J$ contains a ball of certain radius $\gg$, and so for geometric similarity each $J_\gs$ for $\gs\in \C C_k$ contains a ball of radius $\gg |J_\gs|\geq \gg 2^{-k} r$ and clearly this ball is contained in $B'$. Taking into account the disjointness of all such balls belonging to $\C C$, we have
\begin{equation} \label{eq235} \sum_k (\#\C C_k) (\gg 2^{-k} r)^d \leq (2r)^d \te{ which yields } \sum_k (\#\C C_k) 2^{-kd} \leq 2^d \gg^{-d}, \end{equation}
where $\# \C C_k$ is the cardinality of the set $\C C_k$ for all $k\geq 1$. \
Let $h\in (\gq, +\infty)$ given by $P(h)=0$ and assume $h = d$, the dimension of  underlying space. Then using \eqref{eq234} and \eqref{eq235}, we have
\[\mu_h(B(x, r)) \leq \sum_{\gs \in \C C} \mu_h(J_\gs)=\sum_k \sum_{\gs \in \C C_k} |J_\gs|^h< \sum_k (\#\C C_k)2^{-kh} 2^hr^h= \sum_k (\#\C C_k)2^{-kd} 2^hr^h, \ \text{hence}\]
\[\mu_h(B(x, r))< 2^d \gg^{-d} 2^hr^h \leq  4^d \gg^{-d} r^h, \ \text{and} \ \limsup_{r\to 0} \frac{\mu_h(B(x, r))}{r^h}\leq 4^d \gg^{-d}, \]
so by Proposition~\ref{prop22} and since $h = d$, we obtain $\C H^d(E) \geq \frac{1}{4^d \gg^{-d}}>0$. 
\end{proof}

\subsection*{Proof of Theorem~\ref{theorem}}

By Proposition~\ref{prop120} and Proposition~\ref{prop121}, the proof is complete.

$\hfill\square$

\begin{remark} \label{rem11}
For an infinitely generated Moran set with non-stationary coefficients $E$, Theorem~\ref{theorem} guarantees that, if $h\in (\gq, +\infty)$ is the unique number for which $P(h)=0$, then $\C H^h(E)<\infty$; and if $h = d$, then $\C H^h(E)>0$, where $d$ is the dimension of the underlying space. \
However if $h<d$, then it may happen that $\C H^h(E)$ is null. Then since an infinitely generated self-similar set is a special case of our infinitely generated Moran sets with non-stationary coefficients, we can use Example 2.6 in \cite{M} due to M. Moran.

$\hfill\square$
\end{remark}

 From Theorem \ref{theorem} above we know that the dimension of $E(\omega)$ depends only on the frequency vector $\eta$ and the contraction rates given by the infinite vectors $\Psi_i, 1 \le i \le m$, for any $\omega \in A^\mathbb N_\eta$.
We now prove that the Hausdorff dimension of $E(\omega), \ \omega \in A^\mathbb N_\eta$, depends \textbf{real analytically on the frequency vector} $\eta = (\eta_1, \ldots, \eta_m)$, when the contraction rates are fixed.

\begin{theorem}\label{real-an}
In the setting of Section 2 let us fix the infinite positive vectors $\Psi_i, 1 \le i \le m$ satisfying (\ref{gamma}), and denote by $\mathcal{V}$ the interior of the set $\{(\eta_1, \ldots, \eta_m) \in [0, 1]^m, \ \eta_1 + \ldots + \eta_m = 1\}$. Consider now $\eta = (\eta_1, \ldots, \eta_m) \in \mathcal{V} \subset \mathbb R^m$ and $\omega \in A^\mathbb N_\eta$, and let $E(\omega)$ be an associated Moran fractal with contraction rates given by $\{\Psi_i\}_{1 \le i \le m}$. Then the real-valued function   $\Delta: \mathcal{V} \to \mathbb R$, $$\Delta(\eta_1, \ldots, \eta_m):= \text{dim}_\text{H}(E(\omega)),$$ is real analytic on $\mathcal{V}$.
\end{theorem}

\begin{proof}

We fixed the contraction rates given by the infinite vectors $\{\Psi_i = (\gamma_{i1}, \gamma_{i2}, \ldots)\}_{1 \le i \le m}$. Then from Theorem \ref{theorem}, for any $\omega \in A^\mathbb N_\eta$,  it follows that the Hausdorff dimension of the fractal $E(\omega)$ is equal to the unique $t$  satisfying the equation 
\begin{align}
\eta_1 \log \mathop{\sum}\limits_{j \ge 1} \gamma_{1j}^t + \ldots + \eta_m \log \mathop{\sum}\limits_{j \ge 1} \gamma_{mj}^t = 0
\end{align}

Let us now consider the function $W: \mathcal{V} \times (\theta, \infty) \to \mathbb R$ given  by the formula
$$W(\eta_1, \ldots, \eta_m, t):= \eta_1 \cdot \log \mathop{\sum}\limits_{j \ge 1} \gamma_{1j}^t + \ldots + \eta_m \cdot \log \mathop{\sum}\limits_{j \ge 1} \gamma_{mj}^t
$$

Clearly $W(\cdot, t)$ is real-analytic in $(\eta_1, \ldots, \eta_m)$ for any fixed $t$, as it is linear in $\eta_1, \ldots, \eta_m$. Let us now fix $\eta_1, \ldots, \eta_m$ and consider the function $W(\eta_1, \ldots, \eta_m, \cdot)$. Locally near any $t \in (\theta, \infty)$ the series $\mathop{\sum}\limits_{j \ge 1} \gamma_{ij}^t$ is  uniformly convergent; thus it is a real-analytic function on $(\theta, \infty)$ as the limit of a uniformly convergent sequence of real analytic functions (it is also the restriction of a holomorphic map obtained as the limit of a sequence of holomorphic maps).

Now from the definition of $W$ and as $t \in (\theta, \infty)$,  it is easy to see that $W$ is locally bounded in $(\eta_1, \ldots, \eta_m, t)$. Hence since $W$ is separately real analytic and locally bounded, it follows from \cite{Br} (see also \cite{Kr}), that $W$ is real analytic \textit{jointly} in $(\eta_1, \ldots, \eta_m, t)$.

Now as we saw above, for any fixed $(\eta_1, \ldots, \eta_m) \in \mathcal{V}$, there exists a unique zero of the equation $W(\eta_1, \ldots, \eta_m, t) = 0$ and this zero is exactly $\Delta(\eta_1, \ldots, \eta_m)$. Hence $\frac{\partial W}{\partial t} \ne 0$.
Then $$W(\eta_1, \ldots, \eta_m, \Delta(\eta_1, \ldots, \eta_m)) = 0,$$ and from the Real Analytic Implicit Function Theorem (\cite{Kr}), we obtain that $\Delta(\eta_1, \ldots, \eta_m)$ depends real analytically on $(\eta_1, \ldots, \eta_m) \in \mathcal{V}$. 

\end{proof}

Real analyticity in higher dimension is a strong property and we can apply Lojaciewicz Vanishing Theorem (see \cite{Kr}) in order to control the \textbf{oscillation of Hausdorff dimension of} $E(\omega), \ \omega \in A^\mathbb N_\eta$, when varying the frequency vector $\eta$ (with contraction rates being fixed).  \ 

Thus from Theorem \ref{real-an} and the \textbf{Lojaciewicz Vanishing Theorem} we obtain the following estimation from below  of the dimension oscillation, in terms of distances to certain real-analytic subvarieties in $\mathbb R^m$:

\begin{corollary}\label{osc}

In the setting of Theorem \ref{real-an}, it follows that  for any compact set $K \subset \mathcal{V} \subset \mathbb R^m$ there exists a constant $C>0$ and an integer $q$ depending on $K$, such that for any $\eta = (\eta_1, \ldots, \eta_m) \in K$ and any $\omega \in A^\mathbb N_\eta$ we obtain:
\begin{align}\label{Z}
|\text{dim}_\text{H}(E(\omega)) - \text{dim}_\text{H}(E(\rho^0))| \ge C \cdot \text{dist}(\eta, Z_{\rho^0})^q, \ \text{where},
\end{align}
 $$Z_{\rho^0} := \{\omega' \in \mathcal{V}, \ \Delta(\omega') = \text{dim}_\text{H}(E(\rho^0)) \}$$ is a real-analytic subvariety in $\mathbb R^m$, for any fixed stochastic vector $\rho^0 \in \mathcal{V}$.
\end{corollary}

\

\section{Fractals in $\mathbb R^d$ associated to $f$-expansions of real numbers.}

In the sequel we will apply our results in order to associate to (generic) numbers $x \in [0, 1)$, families of infinitely generated non-stationary Moran sets $E(x) \subset \mathbb R^d$ constructed according to \textbf{$f$-expansions} of $x$ (see for eg. \cite{Re}, \cite{A}, \cite{KP} for generalities on $f$-expansions).

\subsection{$m$-ary and beta expansions.}\label{mbeta}
We start with the well-known $m$-ary expansion, where $m \ge 2$ is an integer. In making the connection with our family of sets $\{E(\omega)\}_\omega$ constructed in Section 2, we use results about frequency of digits for the $m$-ary expansion.  
In order to detail the above construction let a  number $x \in [0, 1) \setminus \mathbb Q$; then $x$ has a unique \textbf{$m$-ary expansion}:
\begin{equation}\label{exp1}
x = \mathop{\sum}\limits_{k = 1}^\infty \frac{d_k(x)}{m^k}, \ \ d_k(x) \in \{0, 1, \ldots, m-1\}
\end{equation}

The coefficients $d_k(x)$ will also be denoted by $d_k$ when no confusion about $x$ is possible; they are called \textbf{the digits} of $x$ in the $m$-ary expansion, and are obtained from iterations of the piecewise linear map $T: [0, 1) \to [0, 1), \ 
T(x) = mx \ (\text{mod} \ 1), \ x \in [0, 1)$.
We have therefore $$x = \frac{d_1(x)}{m} + \frac{Tx}{m} = \frac{d_1(x)}{m} + \frac{d_2(x)}{m^2} + \ldots + \frac{d_k(x)}{m^k} + \frac{T^kx}{m^k} $$

It is known that the Lebesgue measure $\lambda$ is $T$-invariant, and that the random variables $X_k$ given by the digits $d_k(\cdot), \ k \ge 1$ are independent identically distributed with respect to the Lebesgue measure on $[0, 1)$ (see for instance \cite{DK}).
Therefore for any $j \in \{0, \ldots, m-1\}$ it follows from the Strong Law of Large Numbers that Lebesgue-a.e number $x \in [0, 1)$ is \textit{normal}, i.e that for Lebesgue-a.e $x \in [0, 1)$ and $ 0 \le j \le m-1$, $$\mathop{\lim}\limits_{n \to \infty}\frac 1n Card\{1 \le k \le n, d_k(x) = j\} = \mathop{\lim}\limits_{n \to \infty} \frac 1n \mathop{\sum}\limits_{i = 1}^n X_i(x) = \lambda(y \in [0, 1), d_1(y) = j) = \frac 1m$$  

Denote by $\mathcal{N}(m)$ the set of normal numbers in the $m$-ary expansion.
Now if $x \in \mathcal{N}(m)$ and $$\bar p := (\frac 1m, \ldots, \frac 1m) \ \text{and} \ \go(x):= (d_1(x), d_2(x), \ldots),$$ then it follows that $\go \in A^{\mathbb N}_{\bar p}$.   \
However there are also other real numbers $x \in [0, 1) \setminus \mathbb Q$ which are \textbf{not} normal, but for which the asymptotic frequency $p_j$ of digit $j$ exists, for every $0 \le j \le m-1$; we call such numbers $(p_0, \ldots, p_{m-1})$-quasinormal, and denote their set by 
$$F(p_0, \ldots, p_{m-1}):=
\Big \{ x = \mathop{\sum}\limits_{k \ge 1} \frac{d_k(x)}{m^k}, \ \frac{Card\{j, \ 1 \le j \le n, \ d_j(x) = i \}}{n} \mathop{\longrightarrow}\limits_{n \to \infty} p_i, \  0 \le i \le m-1\Big \}
$$
From \cite{Be} (for the dyadic case) and \cite{E} (for the general case), we have the dimension formula: $$\text{dim}_\text{H}(F(p_0, \ldots, p_{m-1})) = -\frac{1}{\log m} \mathop{\sum}\limits_{0 \le j \le m-1} p_j \log p_j$$ On the other hand the sets $F(p_1, \ldots, p_m)$ are dense in the unit interval, as can be seen since the frequencies are not affected by the first $n$ digits, for any fixed $n$. \ 
Denote next by $\mathcal{Q}_m$ the union of all sets $F(p_0, \ldots, p_{m-1})$ over all stochastic vectors $(p_0, \ldots, p_{m-1})$; a number from  $\mathcal{Q}_m$ is called \textit{quasinormal} in the $m$-ary expansion. 

Let us fix now a finite collection of infinite positive vectors $\Psi_i = (\gamma_{i1}, \gamma_{i2}, \ldots), 0 \le i \le m-1$, satisfying the conditions (\ref{gamma}) in the construction from Section 2, i.e $$\sum_{j=1}^{\infty}\gamma_{ij} < \infty, 1 \le i \le m \ \text{and} \ 
\sup\set{\gamma_{ij}, 1 \le i \le m,  j\in \D N}<1$$ 
Given the infinite vectors $\Psi_i, \ 0 \le i \le m-1$ as above, we obtain then a functional $E(\cdot; m)$ assigning to every $x \in \mathcal{Q}_m$, a Moran set $E(x; m) \subset \mathbb R^d$ with infinitely many generators and non-stationary contractions as in Section 2, by using the sequence of digits $\omega = (d_1(x), d_2(x), \ldots)$. \
We obtain then a function describing the Hausdorff dimensions of these Moran sets, $$H_m: \mathcal{Q}_m \to [0, \infty), \ H_m(x) := \text{dim}_\text{H}(E(x; m)), \ x \in \mathcal{Q}_m$$

 Another famous expansion where digits take finitely many values, is the \textbf{beta-expansion} for $\beta >1$, $\beta \notin \mathbb Z$. It was introduced by R\'enyi (\cite{Re}) and uses iterations of  the map $$T_\beta: [0, 1) \to [0, 1), \ T_\beta(x) = \beta x \ (\text{mod} \ 1)$$
In this case  the main generating equation is $\beta x = T_\beta(x) + [T_\beta(x)],$ \
where $[y]$ and $\{y\}:= y - [y]$,  denote respectively the \textit{integer part}, and the \textit{fractional part} of the real number $y$. Then we obtain the $\beta$-expansion of $x$, $$x = \frac{d_1(x)}{\beta} + \frac{d_2(x)}{\beta^2} + \ldots,$$ where the digits $d_i(x), i \ge 1$ may be denoted also by $d_i, i \ge 1$ and where $d_i \in \{0, \ldots, [\beta]\}, i \ge 1$.
However the map $T_\beta$ does not preserve Lebesgue measure on $[0, 1)$, unlike the map $T_m$ for $m \in \mathbb Z$. Nevertheless R\'enyi proved in \cite{ Re} that there exists a unique probability measure $\nu_\beta$ equivalent to the Lebesgue measure, and which is $T_\beta$-invariant. The probability $\nu_\beta$ is ergodic with respect to $T_\beta$ (see \cite{Re}, \cite{DK}). Moreover Parry (\cite{Pa}) gave an \textbf{explicit form} for the density function $h_\beta$ of $\nu_\beta$,  $h_\beta(x):= \frac{d\nu_\beta}{d\lambda}(x)$ for Lebesgue-a.e $x$,
\begin{equation}\label{Parry}
h_\beta(x) = \frac{1}{\mathcal{I}(\beta)}\mathop{\sum}\limits_{n \ge 0, \ x< T_\beta^n1} \frac{1}{\beta^n}, \ \text{and} \ \mathcal{I}(\beta) = \int\limits_0^1 \big(\sum\limits_{n \ge 0, \ x< T^n_\beta 1} \frac{1}{\beta^n}\big) \ d\lambda(x)
\end{equation}
Then from the Ergodic Theorem (\cite{Wa}), it follows that for $\nu_\beta$-a.e $x \in [0, 1)$ (hence for Lebesgue-a.e $x \in [0, 1)$) and for $j \in \{0, 1, \ldots, [\beta]\}$, we have:
\begin{equation}\label{beta}
\aligned \mathop{\lim}\limits_{n \to \infty} \frac 1n Card \{1 \le i \le n,  \ d_i = j\} &= \mathop{\lim}\limits_{n \to \infty} \frac 1n \mathop{\sum}\limits_{i= 0}^{n-1}
\chi_{[\frac{j}{\beta}, \frac{j+1}{\beta})}(T_\beta^i(x)) \\
& = \nu_\beta([\frac{j}{\beta}, \frac{j+1}{\beta})) = \eta_j
\endaligned
\end{equation}   
As before we denote the set of numbers $x \in [0, 1)$ for which the above holds, by $\mathcal{N}(\beta)$ and call it the set of \textit{normal} numbers for the $\beta$-expansion; $\mathcal{N}(\beta)$ has full Lebesgue measure in $[0, 1)$. Hence for $x \in \mathcal{N}(\beta)$, the sequence $\omega = (d_1(x), d_2(x), \ldots)$ given by the digits of $x$, belongs to $A^{\mathbb N}_{\eta}$, for the stochastic vector $\eta = (\eta_0, \ldots, \eta_{[\beta]})$ specified in (\ref{beta}). \

For a stochastic vector $(p_0, \ldots, p_{[\beta]})$ we define also the $(p_0, \ldots,  p_{[\beta]})$-\textit{quasinormal numbers}:
$$F(p_0, \ldots, p_{[\beta]}) := \{ x = \mathop{\sum}\limits_{k \ge 1} \frac{d_k(x)}{\beta^k} \ \in [0, 1), \ \frac{Card\{j, \ 1 \le j \le n, \ d_j(x) = i \}}{n} \mathop{\to}\limits_{n \to \infty} p_i, \  0 \le i \le [\beta]\}$$

Given a collection of infinite vectors $\Psi_i = (\gamma_{i1}, \gamma_{i2}, \ldots), 0 \le i \le [\beta]$ satisfying (\ref{gamma}) we can construct infinitely generated non-stationary Moran fractals $E(x; \beta) \subset \mathbb R^d$ for every $x \in F(p_0, \ldots, p_{[\beta]})$, thus obtaining a dimension function $$H_\beta: \mathcal{Q}_\beta \to [0, \infty), \  H_\beta(x) := \text{dim}_\text{H}(E(x; \beta))$$

We obtain now the dimension of $F(p_0, \ldots, p_{[\beta]})$  similarly as in case of the $m$-ary expansion.

\begin{lemma}\label{dim-beta}
Given $\beta >1, \beta \notin \mathbb Z$ and a stochastic vector $(p_0, \ldots, p_{[\beta]})$, it follows that the set of $(p_0, \ldots, p_{[\beta]})$-quasinormal numbers for the $\beta$-expansion, has Hausdorff dimension $$\text{dim}_\text{H}(F(p_0, \ldots, p_{[\beta]})) = - \frac{\mathop{\sum}\limits_{i = 0}^{[\beta]} p_i \log p_i}{\log \beta - p_{[\beta]} \log\{\beta\}}$$
\end{lemma}

\begin{proof}
The idea is to use the mass distribution principle (see for eg. \cite{F}), in order to obtain a probability measure and then to compare it with Lebesgue measure on basic intervals.  
Let us denote by $I_0= [0, \frac 1\beta), \ldots, I_{[\beta]}= [\beta, 1)$ the basic intervals for the expansive map $T_\beta$, and let $I_{i_1 \ldots i_k}:= \{x \in I_{i_1}, T^j(x) \in I_{i_j}, 1 \le j \le k\}, \ k \ge 1$. 

For a given $j \in \{0, \ldots, [\beta]\}$, let us define the random variables $X_i, i \ge 1$ defined by  
$$
X_i(x) = \left\{\begin{array}{ll}
                     1, \ \text{if} \ d_i(x) = j,  \\
                     0, \ \text{if} \ d_i(x) \neq j
                     \end{array}
           \right.
$$
We will take now a probability measure $\mu$ which makes the random variables $X_i, i \ge 1$ independent. This is defined such that $\nu(I_{i_1\ldots i_k}) =  p_{i_1} \ldots p_{i_k}$ for any $i_1, \ldots, i_k$ and  $k \ge 1 $. For this probability we thus have a sequence of independent identically distributed random variables $X_i, i \ge 1$. From the Strong Law of Large Numbers it follows as in \cite{F} that,  if we denote by $n_j(x, k)$ the number of occurences of the digit $j$ in the first $k$ digits of $x$, then $$\mathop{\lim}\limits_{k \to \infty} \frac{n_j(x, k)}{k} = p_j, \ j = 0, \ldots, [\beta]$$
Hence we see that $\mu$ gives measure 1 to the set $F(p_0, \ldots, p_{[\beta]})$. Now let us  notice that the length of the interval $I_{i_1 \ldots i_k}$ is given by: $\lambda(I_{i_1 \ldots i_k}) = \frac{\{\beta\}^{n_{[\beta]}(x, k)}}{\beta^k}$.
Then, if we denote the interval $I_{i_1 \ldots i_k}$ containing $x$ by $I_k(x)$, we have  for any $\delta >0$,
$$\frac 1k \log \frac{P(I_k(x))}{\lambda(I_k(x))^\delta} \mathop{\to}\limits_{k \to \infty} \mathop{\sum}\limits_{0 \le i \le [\beta]} p_i \log p_i + \delta (\log \beta - p_{[\beta]} \log \{\beta\})$$
So from the  mass distribution principle (\cite{F}, \cite{Ma}) it follows that the Hausdorff dimension is: \ $\text{dim}_\text{H}(F(p_0, \ldots, p_{[\beta]})) = - \frac{\mathop{\sum}\limits_{0 \le i \le [\beta]} p_i \log p_i}{\log \beta - p_{[\beta]} \log \{\beta\}}$.

\end{proof}

We obtain then the following result about the strange behaviour of $H_m$ on the set of quasinormal numbers, saying that all possible values of $H_m$ are attained in every small subinterval; the same results are obtained also for $H_\beta$:

\begin{theorem}\label{thm2}
Let an integer $m \ge 2$  and a collection of infinite vectors $\Psi_i = (\gamma_{i1}, \ldots), 0 \le i \le m-1$ satisfying  (\ref{gamma}), and consider the dimension function $H_m(\cdot)$ defined above for the associated infinitely generated non-stationary Moran fractals $E(\cdot; m) \subset \mathbb R^d$ . 

Then, $H_m$ is measurable, everywhere discontinuous on $\mathcal{Q}_m$,  and it is constant on every set of positive dimension $F(p_0, \ldots, p_{m-1})$. Moreover for any $x \in \mathcal{Q}_m$ and any neighbourhood $U$ of $x$, it follows that $H_m$ attains all its possible values inside $U$, i.e $$H_m(\mathcal{Q}_m) = H_m(U \cap \mathcal{Q}_m)$$

The same conclusions hold, for the function $H_\beta$ associated to $\beta >1, \beta \notin \mathbb Z$ and to a fixed collection of infinite vectors $\Psi_i, 0 \le i \le [\beta]$ satisfying (\ref{gamma}). 

\end{theorem}

\begin{proof}

We use Theorem \ref{theorem} to obtain that $H_m(x)$ is the unique zero of the respective pressure function $P(t) = \mathop{\sum}\limits_{i=0}^{m-1} p_i \log \mathop{\sum}\limits_{j \ge 1} \gamma_{ij}^t$, whenever $x \in F(p_0, \ldots, p_{m-1})$.
Hence $H_m$ is constant on the fractal set $F(p_0, \ldots, p_{m-1})$ which,  from the above discussion has positive Hausdorff dimension, which is in fact equal to $-\frac{1}{\log m} \mathop{\sum}\limits_{ j =0}^{m-1} p_j \log p_j >0$.

Moreover from the proof of Proposition 10.1 of \cite{F} it follows that the set $F(p_0, \ldots, p_{m-1})$ is dense in $[0, 1)$, since given any interval $I_{j_1 \ldots j_r}$, we can take a subinterval $I_{j_1 \ldots j_r j_{r+1} \ldots j_n}$ which contains an $r$-preimage $z$ of a point from $F(p_0, \ldots, p_{m-1})$; thus also $z \in F(p_0, \ldots, p_{m-1})$ since the frequency of appearance of $i$ among digits,  is not affected by the first $r$ digits. 

 This means that given any small neighbourhood $U$ of $x \in \mathcal{Q}_m$ and any stochastic vector $(p_0, \ldots, p_{m-1})$, there exist points from $U \cap F(p_0, \ldots, p_{m-1})$, hence all the possible values of $H_m$ are attained inside $U$. \
The  measurability of $H_m$ follows from the measurable dependence of the frequencies of digits on $x$, and from the formula for dimension in Theorem \ref{theorem}. \

For the function $H_\beta$ associated to $\beta >1, \beta \notin \mathbb Z$, we can  do the same argument as above, by using the positivity of dimension for $F(p_0, \ldots, p_{[\beta]})$, proved in Lemma \ref{dim-beta}.

\end{proof}

For the above vectors $\Psi_i, 0 \le i \le m-1$, consider the function $H_m(\cdot)$ on $\mathcal{Q}_m$. We obtain then, from Theorem \ref{thm2} and Corollary \ref{osc}, the following result about the sets of frequencies associated to the level sets of $H_m$, and  about the oscillation of $H_m$: \ if $H_m(x) = H_m(y)$ for numbers $x, y \in \mathcal{Q}_m$, then the frequency vector of $y$ belongs to a \textbf{real-analytic subvariety} $Z_x \subset \mathbb R^m$, which depends on $x$.

\begin{corollary}\label{H-osc}
 In the setting of Theorem \ref{thm2} and Corollary \ref{osc}, it follows that for any compact set of frequencies $K \subset \mathcal{V} \subset \mathbb R^m$, there exists a constant $C(K)>0$ and an integer $q(K)$ such that, if $x \in \mathcal{Q}_m, y \in F(\rho_0, \ldots, \rho_{m-1}) \subset \mathcal{Q}_m$ and  $(\rho_0, \ldots, \rho_{m-1}) \in K$, then: 
$$
| H_m(y) - H_m(x)| \ge C(K) \cdot dist((\rho_0, \ldots, \rho_{m-1}), Z_x)^{q(K)},
$$
where $Z_x = \{\eta \in \mathcal{V}, \Delta(\eta) = H_m(x)\}$ is a real-analytic subvariety in $\mathbb R^m$. 

In particular, if $H_m(y) = H_m(x)$ and if $y \in F(\rho_0, \ldots, \rho_{m-1})$, then the frequency vector $(\rho_0, \ldots, \rho_{m-1})$ must belong to the real-analytic subvariety  $Z_x \subset \mathbb R^m$.  \
 
The same conclusions hold for $H_\beta(\cdot)$, with $ \beta >1, \beta \notin \mathbb Z$.

\end{corollary}

Let us notice that by taking $(\rho_0, \ldots, \rho_{m-1})$ in a fixed compact set $K$, we take only a part of  the possible values for the  frequency vectors; also the constants $C(K), q(K)$ depend on $K$.  The proof for $H_\beta, \beta >1, \beta \notin \mathbb Z$, is similar to the case of $H_m$, $m\ge 2$ integer.

\

\subsection{Fractals $\tilde E_m(x), \tilde E_\beta(x)$ with contraction vectors $\Psi_i(x)$ depending on $x$.}\label{tildeE} 

\

We define now \textbf{another class} of infinitely generated Moran fractals, $\tilde E_m(x)$ whose non-stationary contraction rates $\Psi_i(x)$ are not fixed, and depend on the digits in the $m$-ary expansion of $x$. This is a natural construction, where infinitely many non-stationary contraction rates are needed in order to obtain fractals that intimately  \textbf{depend on all digits of $x$}. It works for $f$-expansions with finitely many digit values, like $m$-ary expansions, $\beta$-expansions, the Bolyai-Renyi expansion, etc. However for $m$-ary and beta-expansions we have explicit formulas for  the absolutely continuous invariant measures and for dimensions of sets $F(\eta)$.

Let us take then an integer $m \ge 2$, an arbitrary stochastic vector $(p_0, \ldots, p_{m-1})$ and the set $F(p_0, \ldots, p_{m-1})$ of $(p_0, \ldots, p_{m-1})$-quasinormal numbers. Consider an arbitrary real number $x \in F(p_0, \ldots, p_{m-1})$, and define the infinite contraction vectors depending on $x$, $$\Psi_i(x) = (\gamma_{i1}(x), \gamma_{i2}(x), \ldots), \ i = 0 , \ldots, m-1,$$ where $\gamma_{ij}(x)$ are given for all  $j \ge 1, 0 \le i \le m-1$, by: 
\begin{equation}\label{gammaij}
\gamma_{ij}(x):= 2^{- 1 - Card\{k, \ 1 \le k \le j, \ d_k(x) = i\}}, \  
\end{equation}
where $d_k(x) $ are the digits of $x$ in its representation in base $m \ge 2$. 

Without loss of generality we will assume that all frequencies $p_i, 0 \le i \le m-1$ are positive (otherwise, the respective digits do not matter in the dimension formula anyway).

Then, using our method in Section 2 we construct a fractal set $\tilde E_m(x)$ corresponding to the sequence of digits of $x$;  this fractal has non-stationary contraction rates given by the contraction vectors $\Psi_i(x), 1 \le i \le m$, and frequency vector $(p_0, \ldots, p_{m-1})$. 

We see in the next Corollary that the \textbf{behavior of the dimensions} of the fractals $\tilde E_m(x)$, is \textbf{different} than in the case of the fractals of type $E(x; m)$; now these dimensions do not remain constant even when $x$ varies in a set $F(\eta)$, with $\eta$ fixed.

\begin{corollary}\label{cortildeE}
 \ a) Let an integer $m \ge 2$ and an arbitrary stochastic vector $\eta = (p_0, \ldots, p_{m-1})$ so that $p_i >0, 0 \le i \le m-1$. Consider the set $F(p_0, \ldots, p_{m-1})$ of $(p_0, \ldots, p_{m-1})$-quasinormal numbers, let a number $x \in F(p_0, \ldots, p_{m-1})$, and construct the infinitely generated non-stationary fractal $\tilde E_m(x)$. Then the function $\delta_{m, \eta}: F(\eta) \to [0, \infty)$ given by:
$$\delta_{m, \eta}(x) := \text{dim}_{\text{H}}(\tilde E_m(x)), \ x \in F(\eta)$$
is continuous and, for every $x \in F(\eta)$ we have 
$$\mathop{\sum}\limits_{i = 0}^{m-1} p_i \cdot \log \mathop{\sum}\limits_{j \ge 1}   
2^{- (1+ Card\{k, \ 1 \le k \le j, \ d_k(x) = i\})  \cdot \delta_{m, \eta}(x)} = 0$$
This applies in particular for normal numbers $x\in \mathcal{N}(m)$, with their frequencies $(\frac 1m, \ldots, \frac 1m)$.

\ b) Similar result holds also for $\beta$-expansions, $\beta >1, \beta \notin \mathbb Z$.
In particular it applies  to normal numbers $x \in \mathcal{N}(\beta)$, with their frequencies  $(\eta_0, \ldots, \eta_{[\beta]})$ given by (\ref{Parry}) and (\ref{beta}).
 \end{corollary}

\begin{proof}

\ a) We assume that the frequencies of the digits of $x$ are given by the stochastic vector $\eta = (p_0, \ldots, p_{m-1})$. Let us also denote by $$R_i(x, j):= Card\{k, \ 1 \le k \le j, \ d_k(x) = i \}, \ 0 \le i \le m-1, j \ge 1$$  Then since $x \in F(p_0, \ldots, p_{m-1})$, it follows that $$\lim_{j \to \infty} \frac{R_i(x, j)}{j} = p_i >0, \ \text{for} \ 0 \le i \le m-1$$ 
So there exists an integer $N$ and some $\alpha >0$ such that for $j >N$ and $ 0 \le i \le m-1$ we have $R_i(x, j) > \alpha \cdot j$. 
This implies then that $\gamma_{ij}(x) < 2^{-\alpha j}, \ 0 \le i \le m-1, j >N$.
Hence for any $0 \le i \le m-1$ we obtain, $$\mathop{\sum}\limits_{j \ge 1} \gamma_{ij} < \mathop{\sum}\limits_{1 \le j \le N} \gamma_{ij} + \mathop{\sum}\limits_{j > N} (\frac{1}{2^\alpha})^j < \infty $$
Also from its definition above, we have that $\gamma_{ij}(x) \le \frac 12, \  i, j \ge 1$. Therefore the contraction rates $\gamma_{ij}$ satisfy the conditions (\ref{gamma}). Then,  by the method in Section 2, we can construct the fractal $\tilde E_x(m)$ with non-stationary contraction rates $\gamma_{ij}(x)$ and frequency vector $(p_0, \ldots, p_{m-1})$.

By Theorem \ref{theorem} we know that the Haussdorff dimension of $\tilde E_x(m)$ is given as the zero of the pressure function and thus we have the formula:
$$\mathop{\sum}\limits_{i = 0}^{m-1} p_i \cdot \log \mathop{\sum}\limits_{j \ge 1}   
2^{- \delta_{m, \eta}(x) \cdot (1+R_i(x, j))} = 0$$

One notices that the dimension function $\text{dim}_\text{H}(\tilde E_m(x))$ is not constant when $x$ varies in $F(\eta)$; in fact the first digits in the expansion of $x$, $d_k(x)$ for $k$ small, have a large influence on this dimension. The fact that the set $F(\eta)$ is large enough is guaranteed by the fact that $\text{dim}_\text{H}(F(\eta)) >0$, from the discussion in \ref{mbeta}. 
The continuity of $\delta_{m, \eta}(\cdot)$ on $F(\eta)$ follows from the fact that $\varepsilon$-close irrational numbers $x, x'$ have the same digits in their $m$-expansion up to a certain large order $N(\varepsilon)$. This implies that the contraction rates $\gamma_{ij}(x)$ and $\gamma_{ij}(x')$ coincide up  to $N(\varepsilon)$, and on the other hand the tail series  $\mathop{\sum}\limits_{j \ge N(\varepsilon)}   
2^{- \delta \cdot (1+ R_i(x, j))}$ are $\varepsilon$-small. 

For the $m$-ary expansion, the digits of any normal number $x \in \mathcal{N}(m)$ have equal frequencies $(\frac 1m, \ldots, \frac 1m)$, and we can apply the above result.

\ b)  We can also construct non-stationary fractals $\tilde E_\beta(x)$ with contraction vectors $\Psi_i(x)$ that depend on all the digits of $x$, for quasinormal numbers $x$ in the $\beta$-expansion, for any $\beta>1, \beta \notin \mathbb Z$. In this case again the dimension of $\tilde E_\beta(x)$ varies continuously when $x$ ranges in a set of type $F(\eta)$, and one can consider the level sets of this dimension function. Moreover as showed in Lemma \ref{dim-beta}, the dimension of $F(\eta)$ is positive for $\beta$-expansions too. 

For the beta-expansion, any normal number $x \in \mathcal{Q}(\beta)$ has asymptotic frequencies given by the Parry formula (\ref{Parry}) applied in the Ergodic Theorem (\ref{beta}).
 
\end{proof}

\begin{remark}\label{remtilde}

By using the fractals $\tilde E_m(x) \subset \mathbb R^d$ associated to the $m$-ary expansion of $x$ and considering the dimension of $\tilde E_m(x)$, one can decompose the set of $(p_0, \ldots, p_{m-1})$-quasinormal numbers $F(p_0, \ldots, p_{m-1})$, into level sets of the  function $x \mapsto \text{dim}_\text{H}(\tilde E_m(x)), \ x \in F(p_0, \ldots, p_{m-1})$.
In this way, we can \textbf{distinguish} between various numbers $x \in [0, 1)$ which have the same asymptotic frequency vector $(p_0, \ldots, p_{m-1})$, but which have \textbf{different starting digits} $(d_1(x), d_2(x), \ldots, d_N(x))$, $N$ fixed. \
In particular Corollary \ref{cortildeE} can be applied to the respective sets of normal numbers in the $m$-ary or in the $\beta$-expansion.
$\hfill\square$
\end{remark}

\begin{remark}
In subsection \ref{f-exp} we shall study also equilibrium measures $\mu_\phi$ for piecewise H\"older potentials for more general $f$-expansions. Corollary \ref{cortildeE} can be proved also for $f$-expansions with finitely many values for the digits, such as the $m$-ary expansion, $\beta$-expansion and the Bolyai-Renyi expansion, and for stochastic vectors $\eta(\phi) = (\mu_\phi(I_1), \ldots, \mu_\phi(I_m))$, where $m$ is the maximal number of possible values for digits in that $f$-expansion.
$\hfill\square$
\end{remark}

We mention also that, a different type of construction of Moran fractals using frequencies in $m$-ary expansion was done in \cite{LD}. However, our class of fractals is completely different. First, our fractal $E(\omega)$ is modelled by a fixed sequence $\omega \in A^\mathbb N_\eta$, while their method considers projection of the set of all such sequences. Second, our method uses non-stationary contraction rates. Also our sets are infinitely generated, while theirs are finitely generated. Our construction is geared towards sequences of digits $\omega = (d_1(x), \ldots)$ of individual numbers $x$.

\subsection{The case of  countable sets of frequencies.}\label{infreq}
\ \ 
We will extend now Theorem \ref{theorem} to the case when we have an \textbf{infinite stochastic vector} of frequencies $\eta = (\eta_1, \eta_2, \ldots)$, i.e $\eta_i > 0, i \ge 1 \ \text{and} \  \mathop{\sum}\limits_i \eta_i = 1$. We have also an infinite collection of positive vectors $\Psi_i = (\gamma_{i1}, \ldots), i \ge 1$ for the non-stationary contraction rates. Assume that 
\begin{equation}\label{infinite}
\sup\{ \mathop{\sum}\limits_{j \ge 1} \gamma_{ij}, i \ge 1\} < \infty, \ \text{and} \ \sup\{\gamma_{ij}, i, j \ge 1\} < 1
\end{equation}

 Then clearly there exists $t \ge 0$ such that 
$
\sup \{\mathop{\sum}\limits_{j \ge 1} \gamma_{ij}^t, \ i \ge 1\} < \infty
$, and  denote by $\theta$ the infimum of the set of $t$'s with this property. 

We can construct as in Section 2, a countably generated non-stationary  Moran set associated to: $\Psi_i, i \ge 1$,  $\eta = (\eta_1, \eta_2, \ldots)$, a countable set $A = (a_i)_{i \ge 1}$,  and  an infinite vector $\omega = (\omega_1, \omega_2, \ldots) \in A^\mathbb N$ which satisfies the condition: $$\mathop{\lim}\limits_{n \to \infty} \frac{Card\{1 \le k \le n, \ \omega _k = a_i\}}{n} = \eta_i, \  i \ge 1$$ 

The convergence of the sequence \ $c_k(t) := \mathop{\sum}\limits_{i \ge 1} \frac{||\omega|_k||_{a_i}}{k} \log \mathop{\sum}\limits_{j \ge 1} \gamma_{ij}^t, \ k \ge 1,$  follows as in 2.5, with the exception that now we have the inequalities:
\begin{equation}\label{infc}
\aligned
&|c_n(t) - c_p(t)| = \left|\mathop{\sum}\limits_{i = 1}^N \left(\frac{||\omega|_n||_{a_i}}{n} - \frac{||\omega|_p||_{a_i}}{p}\right) \log \mathop{\sum}\limits_{ j = 1}^\infty \gamma_{ij}^t + \mathop{\sum}\limits_{i = N+1}^\infty \left(\frac{||\omega|_n||_{a_i}}{n} - \frac{||\omega|_p||_{a_i}}{p}\right) \log \mathop{\sum}\limits_{ j = 1}^\infty \gamma_{ij}^t\right|\\
&\le \mathop{\sum}\limits_{i = 1}^N \left|\frac{||\omega|_n||_{a_i}}{n} - \frac{||\omega|_p||_{a_i}}{p}\right| \cdot  |\log \mathop{\sum}\limits_{j = 1}^\infty \gamma_{ij}^t| + \left(1-\mathop{\sum}\limits_{i= 1}^N \frac{||\omega|_n||_{a_i}}{n}\right) M + \left(1-\mathop{\sum}\limits_{i = 1}^N \frac{||\omega|_p||_{a_i}}{p}\right) M,
\endaligned
\end{equation}
where $M \in (0, \infty)$ is such that $\left|\log \mathop{\sum}\limits_{ j = 1}^\infty \gamma_{ij}^t\right|  < M, \ \forall i \ge 1$.
Then given $\varepsilon >0$ we obtain from (\ref{infc}) that there exists $N = N(\varepsilon)$ such that $|1- \mathop{\sum}\limits_{1 \le i \le N} \eta_i | < \varepsilon$. Then there exists $n(\varepsilon)$ such that for every $1 \le i \le N$ and every $n > n(\varepsilon)$, we have $\left|\frac{||\omega|_n||_{a_i}}{n} - \eta_i\right| < \varepsilon$. \
Hence for any integers $n, p > n(\varepsilon)$ we obtain from above that
$|c_n(t) - c_p(t)| \le 4 \varepsilon M$. \
We can then obtain the pressure function $P(\cdot)$ defined on $(\theta, \infty)$ as: 
\begin{equation}\label{infpressure}
P(t) = \mathop{\lim}\limits_{k \to \infty} c_k(t) = \mathop{\sum}\limits_{i = 1}^\infty \eta_i \log \mathop{\sum}\limits_{j = 1}^\infty \gamma_{ij}^t
\end{equation}
As in Section 2 and by using (\ref{infc}), we can show that the pressure function $P: (\theta, \infty) \to \mathbb R$ is strictly decreasing, continuous and convex on $(\theta, \infty)$ and that it has a unique zero $h \in (\theta, \infty)$. \
Then, given a sequence of infinite vectors $\Psi_i, i \ge 1$, a stochastic infinite vector of frequencies $\eta = (\eta_1, \ldots)$, and some $\omega \in A^\mathbb N_\eta$, we can construct the countably generated non-stationary Moran fractal $E(\omega)$ as in Section 2. \ 

Now, let $h$ to be the unique zero of the above pressure function and consider an arbitrary $0 < s_* < h$. 
Then there exists an $\alpha = \alpha(s_*)>0$ such that \ $\sum\limits_{i=1}^\infty \eta_i \log \sum\limits_{j \ge 1} \gamma_{ij}^{s_*} > \alpha >0$. \ Also from our assumption in (\ref{infinite}), there exists $M \in (0, \infty)$ satisfying  $|\log \mathop{\sum}\limits_{j=1}^\infty \gamma_{ij}^{s_*}| \le M$, for $ i \ge 1$. 
Since $\mathop{\sum}\limits_{i \ge 1} \eta_i = 1$, this series has tails that go to 0, hence there exists $\tilde N(s_*)$ such that $\sum\limits_{i > \tilde N(s_*)} \eta_i < \frac{\alpha}{4M}$.  Moreover assume without loss of generality that, for the same $\tilde N(s_*)$, we have \ $\sum\limits_{i=1}^{\tilde N(s_*)} \eta_i \log \sum\limits_{j \ge 1} \gamma_{ij}^{s_*} > \frac{3\alpha}{4}$. \ 
Now for any $i \ge 1$, we know that $\lim_{n \to \infty}\frac{||\omega|_n||_{a_i}}{n} = \eta_i$, hence there exists an integer $N(s_*)$ such that for all $n > N(s_*)$, we obtain $$\big|\sum\limits_{i=1}^{\tilde N(s_*)}\frac{||\omega|_n||_{a_i}}{n} - \sum\limits_{i=1}^{\tilde N(s_*)} \eta_i \big| < \frac{\alpha}{4M},  \ \text{and} \ 
\sum\limits_{i=1}^{\tilde N(s_*)} \frac{||\omega|_n||_{a_i}}{n} \log \sum\limits_{j \ge 1} \gamma_{ij}^{s_*} >\frac{2\alpha}{3}$$

From the above inequalities, we obtain that 
\begin{align*}
|\sum\limits_{i = \tilde N(s_*)}^\infty \frac{||\omega|_n||_{a_i}}{n} \log &\sum\limits_{j \ge 1} \gamma_{ij}^{s_*}| \le M \cdot \sum\limits_{i = \tilde N(s_*)}^\infty \frac{||\omega|_n||_{a_i}}{n} = M \cdot (1- \sum\limits_{i = 1}^{\tilde N(s_*)} \frac{||\omega|_n||_{a_i}}{n})\\
&\le M \cdot (|\sum\limits_{i=1}^{\tilde N(s_*)} \frac{||\omega|_n||_{a_i}}{n} - \sum\limits_{i=1}^{\tilde N(s_*)} \eta_i|  + \sum\limits_{i > \tilde N(s_*)} \eta_i) < \frac{\alpha}{2}
\end{align*}

Therefore we conclude that for any integer $n > N(s_*)$, 
$$
\sum\limits_{i = 1}^\infty \frac{||\omega|_n||_{a_i}}{n} \log \sum\limits_{j \ge 1} \gamma_{ij}^{s_*} >0, \ \text{hence} \ 
\mathop{\sum}\limits_{\sigma \in D_n} c_\sigma^{s_*} >1$$
Similarly for $s^* > h$ there exists $N(s^*)$ such that for any $n > N(s^*)$ we have $\sum\limits_{\sigma \in D_n} c_\sigma^{s^*} < 1$.  
In the same fashion we can reprove the results 3.2-3.8 for the case of countably many contraction vectors $\Psi_i, i \ge 1$. Notice that the sets of type $J_\sigma, \sigma \in D_n$ form a cover with disjointed interiors of the fractal $E(\omega)$. This implies that in this case the analogue of Theorem \ref{theorem} will hold true.

\subsection{Fractals determined by asymptotic frequencies of digits in continued fractions expansions.} To apply the last construction for the case of countable values for the digits, we consider the \textbf{continued fraction} expansion. 
In this case  $T: [0, 1) \to [0, 1)$ is given by $$T(x) = \frac 1x - \left[\frac 1x\right], \ x \ne 0, \text{and} \ T(0) = 0$$
Any irrational number  $x \in [0, 1)$ can be represented uniquely in its continued fraction form as
$$x = \frac{1}{d_1(x) + \frac{1}{d_2(x) + \frac{1}{d_3(x) + \ldots}}}$$
The integers $d_k(x)$ are the digits of $x$ in its continued fraction expansion and are defined by: $d_n(x) := \left[\frac{1}{T^{n-1}(x)}\right], n \ge 1$. \  The map $T$ does not preserve the Lebesgue measure $\lambda$, but there exists 
a $T$-invariant ergodic measure $\mu_G$, i.e the \textit{Gauss measure} which is absolutely continuous with respect to $\lambda$ (see for eg. \cite{DK}). It is well-known that the Gauss measure satisfies  $$\mu_G(A) = \frac{1}{\log 2} \int_A \frac{1}{1+x} dx$$
It follows that we can apply the Ergodic Theorem for $T$ and $\mu_G$, and obtain the set of normal numbers for the continued fraction expansion $\mathcal{N}_{G} \subset [0, 1)$ which has $\mu_G$-measure 1 (and thus also Lebesgue measure 1). So  for any $x \in \mathcal{N}_G$ and $k \ge 1$ we have $$\frac 1n Card\{1 \le i \le n, d_i(x) = k\} \mathop{\to}\limits_{n \to \infty} \mu_G(x \in [0, 1), d_1(x) = k):= p_{G, k}$$

Let us consider now an infinite stochastic vector $\eta = (\eta_1, \eta_2, \ldots)$ and consider, as in \cite{KPW}  the probability $\nu_\eta$ which makes the digits $\{X_i\}_i$ in the continued fraction expansion to be independent random variables.  Namely  if the probability measure $\mu$ is defined by $\mu(d_n = j) = \eta_j, \ \forall n, j \ge 1,$ then the probability measure 
\begin{equation}\label{nup}
\nu_{\eta}(A) := \mu(X \in A),
\end{equation}
gives the distribution of the random variable $X = \frac{1}{d_1 + \frac{1}{d_2+ \frac{1}{ \ldots}}}$. \
The measure $\nu_\eta$ is useful in our case since it gives the asymptotic frequencies of appearance of digits for quasinormal numbers. Clearly $\nu_\eta$ is singular with respect to the Lebesgue measure if $\eta \ne \textbf{p}_G:= (p_{G, 1}, p_{G, 2}, \ldots)$. For a stochastic vector $\eta = (\eta_1, \eta_2, \ldots)$ define the set of $\eta$-\textit{quasinormal numbers} $F(\eta)$ to be $$F(\eta) := \{y \in [0, 1), \frac{Card\{1 \le i \le n, \ d_i(y) = k\}}{n} \mathop{\to}\limits_{k \to \infty} \eta_k, \ k \ge 1\}$$
From the Ergodic Theorem applied to the $T$-invariant ergodic probability $\nu_\eta$ (or from the Strong Law of Large Numbers), we know that  $\nu_\eta(F(\eta)) = 1$.
As before we denote by $\mathcal{Q}_G$ the union of all sets $F(\eta)$ over stochastic vectors $\eta$ with $\left|\mathop{\sum}\limits_{i \ge 1} \eta_i \log \eta_i\right| < \infty$, and call it the set of \textit{quasinormal numbers} for the continued fraction expansion.
In this setting,  Kinney and Pitcher proved in \cite{KP} that the dimension of the probability $\nu_\eta$ is given by:
$$\text{dim}_\text{H}(\nu_\eta) = \frac{-\mathop{\sum}\limits_{i \ge 1} \eta_i \log \eta_i}{2 \int_0^1 |\log x| \ d\nu_\eta(x)} >0,$$
hence from the fact that $F(\eta)$ has $\nu_\eta$-measure equal to 1, it follows that 
\begin{equation}\label{HFE}
\text{dim}_\text{H}(F(\eta)) >0
\end{equation}
In \cite{KPW} it was showed that there exists a constant $\varepsilon_0>0$ with $\text{dim}_\text{H}(\nu_\eta) \le 1- \varepsilon_0$.

Let us now take a \textbf{countable collection} of fixed infinite vectors $\Psi_i = (\gamma_{i1}, \gamma_{i2}, \ldots), i \ge 1$ satisfying (\ref{infinite}).  Then, for any stochastic infinite vector $\eta = (\eta_1, \eta_2, \ldots)$ satisfying $\left|\mathop{\sum}\limits_{i \ge 1} \eta_i \log \eta_i\right| < \infty$ and any quasinormal number $x \in F(\eta)$, we can construct an associated infinitely generated fractal $E(\omega) \subset \mathbb R^d$ with non-stationary contraction rates given by $(\Psi_i)_{i \ge 1}$ and $\eta$,  where $\omega := (d_1(x), d_2(x), \ldots)$. This set $E(\omega)$ will be denoted also by $E(x; G)$. 
Then the associated dimension function for the  continued fraction expansion, is defined by
$$H_G: \mathcal{Q}_G \to [0, \infty), \ H_G(x) := \text{dim}_\text{H}(E(x; G))$$
We can now prove a Theorem similar to the one in the $\beta$-expansion case:

\begin{theorem}\label{continued}
Consider a sequence of infinite positive vectors $\Psi_i, i \ge 1$ satisfying condition (\ref{infinite}), and the associated infinitely generated non-stationary Moran fractals $E(x; G) \subset \mathbb R^d,  \ x \in \mathcal{Q}_G$. Then the function $H_G$ is constant on dense subsets of positive Hausdorff dimension,  and for any $x \in \mathcal{Q}_G$ and any interval $U \subset [0, 1), \ x\in U$, it follows that $$H_G(\mathcal{Q}_G) = H_G(U \cap \mathcal{Q}_G)$$ 
\end{theorem}

\begin{proof}
We have from (\ref{HFE}) that the dimension of $F(\eta)$ is positive, hence there exist numbers $x \in [0, 1)$ for which the digits $1, 2, \ldots$ appear with  frequencies  $\eta_1, \eta_2, \ldots$ in the continued fraction expansion of $x$. Then if $x$ can be written in its continued fraction expansion as $$x = [d_1(x), d_2(x), \ldots],$$ and if $y = [d_1(y), d_2(y), \ldots] \in \mathcal{Q}_G$, we can take for any $n>1$,  the real number $y_n = [d_1(y), d_2(y), \ldots, d_n(y), d_{n+1}(x), d_{n+2}(x), \ldots]$. Then $y_n$ is close to $y$ and also the frequencies of digits in the continued fraction expansion of $y_n$, are those of $x$;  hence $y_n \in F(\eta)$. Hence $F(\eta)$ is dense in $\mathcal{Q}_G$; note that $F(\eta)$ has $\nu_\eta$-measure 1, but has Lebesgue measure zero if $\eta \neq \textbf{p}_G$.

Now we use the discussion in \ref{infreq} in order to extend Theorem \ref{theorem} to the case of countable set of frequencies. By using the same arguments we see that the proof extends to the setting of the Moran fractals $E(x; G)$ associated to $\eta$ and to the sequence of infinite vectors $\Psi_i, i \ge 1$. We assume that $x \in \mathcal{Q}_G$. 
Hence, $\text{dim}_\text{H}(E(x; G)) = h(\eta),$
where $h(\eta)$ is the unique zero of the pressure function $P(t)$, namely
$
\mathop{\sum}\limits_{i \ge 1} \eta_i \log \mathop{\sum}\limits_{j \ge 1} \gamma_{ij}^{h(\eta)} = 0$.
Since each of the sets $F(\eta)$ was shown above to be dense in $\mathcal{Q}_G$, and since $H_G$ is constant on $F(\eta)$,  it follows that given any interval $U$ containing $x$, all the possible values of $H_G$ are attained inside $U$. This finishes the proof.

\end{proof}

We will give in \ref{f-tilde} also a class of fractals $\tilde E_G(x)$ with contraction vectors $\Psi_i(x)$ determined by the continued fraction expansion of $x$, similar to the class introduced in \ref{tildeE}.

\subsection{The case of non-stationary fractals associated to other $f$-expansions.}\label{f-exp} \

General $f$-expansions have been studied by many authors and are important in the ergodic theory of numbers (for eg. \cite{A}, \cite{He}, \cite{KP}, \cite{KPW}, \cite{Re}, etc).  Our construction of infinitely generated non-stationary fractals works also by using digits of numbers in $f$-expansions.

Let first an integer $M \ge 2$ or $M = \infty$, and assume $f$ is either a strictly decreasing continuous function on $[1, M+1]$ with $f(1) = 1, f(M+1) = 0$, or $f$ is strictly increasing continuous on $[0, M]$ with $f(0) = 0, f(M) =1$. In the $f$-expansion, one wishes to represent real numbers by a repeated application of $f$ and extraction of integer parts at each step; clearly one has the identities $x = f(f^{-1}(x))$ and $f^{-1}(x) = [f^{-1}(x)] + \{f^{-1}(x)\}$. Thus one can define inductively the \textbf{digits} $d_k(x)$ and the remainders $r_k(x)$ of $x$, as follows:
$$d_0(x) = 0, \ r_0(x) = x, \ d_{k+1}(x) = [f^{-1}(r_k(x))], \ r_{k+1}(x) = \{f^{-1}(r_k(x))\}, \ k \ge 0$$

Then for any $x \in [0, 1)$, the series $$f(d_1(x) + f(d_2(x) + f(d_3(x) + \cdots ) \cdots ))),$$ converges to $x$, if certain technical conditions are satisfied (see \cite{Re}) namely: 

\ \ (C) \ \ in case $f$ is strictly decreasing, there exists $\alpha \in (0, 1)$ so that $|f(t) - f(s)| \le \alpha |t-s|$ if $1+f(2) < s<t$ and $|f(t) - f(s)| \le |t-s|$ if $0 \le s <t$; and,  in case $f$ is strictly increasing, we should have $|f(t) - f(s)| < |t-s|, \ 0 \le s <t$.

Now let us  take the associated transformation of $[0, 1)$, $$T(x) := f^{-1}(x) - [f^{-1}(x)], \ x \in [0, 1),$$ which is injective on the intervals $I_k:=f(k, k+1), 1 \le k \le M$ when $f$ decreases, or $0 \le k \le M-1$ when $f$ increases.  \ 

Assume the map $T$ satisfies the following regularity and expansiveness conditions (\cite{Wa1}):
 
\ i) $T|_{I_k}$ is $\mathcal{C}^2$, for all $k$;

\ ii) there exists some $p \in \mathbb N$ such that $\mathop{\inf}\limits_{x \in I_k, k \ge 1} |(T^p)'(x)| >1$;

\ iii) $\mathop{\sup}\limits_{x, y, z \in I_k, k \ge 1} |\frac{T''(x)}{T'(y) T'(z)}| < \infty$.

Then by using the Perron-Frobenius operator, Walters showed in \cite{Wa1} that $T$  has a unique absolutely continuous invariant probability $\mu_T$, which is in fact the equilibrium measure of the potential $-\log|T'|$. Moreover,  the probability $\mu_T$ is exponentially mixing (\cite{He}).  
In particular when $f(x) = \frac 1x$, we obtain the continued fraction expansion and its Gauss measure $\mu_G$. 

For general $f$-expansions consider now a potential $\phi$ on $\mathop{\cup}\limits_k I_k$, satisfying the growth conditions

\ a1) $\mathop{\sum}\limits_{y \in T^{-1}(x)} e^{\phi(y)} \le K_{T, \phi} < \infty, \ x \in \mathop{\cup}\limits_K I_k$;

\ a2) there exists some $\gamma >0$ and constants $C_k, k \ge 1$ such that $|\phi(x) - \phi(y)| \le C_k |x-y|^\gamma, \ x, y \in I_k$ and $\mathop{\sup}\limits_{x \in I_k, k \ge 1} C_k |T'(x)|^{-\gamma} < \infty$.

Then under conditions i), ii), iii), a1), a2), (C),  it follows from \cite{Wa1} that there exists a unique $T$-invariant equilibrium measure $\mu_\phi$ corresponding to $\phi$, which is also atomless and positive on open nonempty sets. \ If we now consider the infinite stochastic vector $$\eta(\phi) = (\mu_\phi(I_1), \mu_\phi(I_2), \ldots),$$ then from Ergodic Theorem applied to $T$ and $\mu_\phi$,  there exists a Borel set $F(\phi)\subset [0, 1)$ with $\mu_\phi(F(\phi)) = 1$ s.t every $x \in F(\phi)$ has asymptotic frequency vector $\eta(\phi)$ in its $f$-expansion.

Let us assume now that the contraction vectors $\Psi_i = (\gamma_{i1}, \gamma_{i2}, \ldots), \ i \ge 1$ are \textit{fixed}. 

Then, as in Section 2 we obtain for $x \in F(\phi)$ a \textbf{family of Moran fractals} $E(x; f, \phi) \subset \mathbb R^d$, constructed according to the digits of $x$ in its $f$-expansion.  
Thus from Theorem \ref{theorem}, for $x \in F(\phi)$, the Hausdorff dimension of $E(x; f, \phi)$ is equal to the unique number  $h(\phi)$ satisfying:
\begin{equation}\label{dim-phi}
\mathop{\sum}\limits_{i \ge 1} \mu_\phi(I_i) \cdot \log \mathop{\sum}\limits_{j \ge 1} \gamma_{ij}^{h(\phi)} = 0
\end{equation}

Notice also that, in general, for \textit{any} stochastic vector $\eta = (\eta_1, \eta_2, \ldots),$ one can introduce a probability $\nu_\eta$ as follows: \ first consider the measure $\mu_\eta$ for which $d_i(\cdot)$ are i.i.d. random variables and $\mu_\eta(d_n(\cdot) = k) = \eta_k, \ \forall n, k \ge 1$. Then let the measure $\nu_\eta$ defined  as the $\mu_\eta$-distribution of the random variable $X$ defined as  $$X(\cdot) := f(d_1(\cdot) + f(d_2(\cdot) + f(d_3(\cdot) + \ldots)))$$ 
Hence from the definition of $\nu_\eta$, we have for any arbitrary Borel set $B \subset [0, 1)$, that $\nu_\eta(B) = \mu_\eta(X \in B)$.   With respect to this measure $\nu_\eta$, the digits $d_i(\cdot)$ in the $f$-expansion, become independent random variables. \ 
However, the control on the set of $x$'s with frequency vector $\eta$ in their $f$-expansions is weaker than in the case when $\eta$ is given by  $\mu_\phi$. 

Equilibrium measures $\mu_\phi$ enjoy many metric and statistical properties, including in relation to dimension (for eg. \cite{Bo}, \cite{P}, \cite{Wa}, \cite{Wa1}, \cite{Mi}, \cite{Mi1}, etc).  
We obtain again a dimension function $H_f(x)= \text{dim}_\text{H}(E(x; f))$ on the set of quasinormal numbers $x$ in the $f$-expansion, which can be studied as before.

\subsection{Fractals $\tilde E_f(x)$ with vectors $\Psi_i(x)$ determined by the digits of $x$ in $f$-expansion.}\label{f-tilde} \ \ 
We introduce now another family of non-stationary fractals $\tilde E_f(x)$ whose infinite contraction vectors $\Psi_i(x), i \ge 1$ \textbf{depend on all the digits} $d_k(x)$ of $x$ in its $f$-expansion, by contrast to the sets in \ref{f-exp} which depended on fixed contraction vectors and on asymptotic frequencies of digits. Using the sets $\tilde E_f(x)$, we can obtain information about  the digits of $x$, for eg. related to the speed of convergence of partial frequencies of digits towards asymptotic frequencies.

We consider then a function $f$ satisfying conditions (C), i), ii), iii)  from \ref{f-exp}, and a potential $\phi$ on $\mathop{\cup}\limits_{k\ge 1} I_k$ satisfying  a1), a2); let also $\mu_\phi$ be the $T$-invariant equilibrium measure of $\phi$. We have two cases, according to whether the set of digit values $\mathcal{D}$ is finite or infinite.

If the digits in the $f$-expansion take values in a \textbf{finite set} $\mathcal{D}$, then we saw in \ref{tildeE} that we can construct fractals $\tilde E_f(x)$ such that the dimension of $\tilde E_f(x)$ is not necessarily constant, but varies continuously when $x$ ranges in a set of type $F(\eta)$ of $\eta$-quasinormal numbers. In particular we can choose $\eta = \eta(\phi)$ for a potential $\phi$ as above. Also we can prove Corollary \ref{cortildeE} for such fractals $\tilde E_f(x)$. 
Besides the $m$-ary and the $\beta$-expansions discussed in \ref{tildeE}, another $f$-expansion with finite $\mathcal{D}$ is the \textit{Bolyai-Renyi expansion}. 
The Bolyai-Renyi expansion is given by the increasing function $f_B(x) = \sqrt[m]{x+1} -1, \ x\in [0, 2^m-1]$, for an integer $m \ge 2$ fixed; then for a.e $x\in [0, 1)$ we have its representation as $$x = -1 + \sqrt[m]{d_1(x) + \sqrt[m]{d_2(x) + \sqrt[m]{d_3(x)+ \ldots }}},$$  where the digits $d_i(x), i \ge 1$ belong to the finite set $\mathcal{D} = \{0, 1, \ldots, 2^m -2\}$. The conditions from \ref{f-exp}  are satisfied for $f_B$ and its associated map $T_B: [0, 1) \mapsto [0, 1)$, $T_B(x) = \{(x+1)^m\}$;  hence there exists an absolutely continuous $T_B$-invariant probability $\mu_{T_B}$ in this case. \ 

Now, returning to the case of an $f$-expansion with finite digit set $\mathcal{D}$ and associated map $T: [0, 1) \mapsto [0, 1)$, we can form as in \ref{tildeE} the infinitely generated non-stationary fractals $\tilde E_f(x) \subset \mathbb R^d$ associated to quasinormal numbers $x \in F(\eta)$, for some arbitrary probability vector $\eta$. The contraction rates of $\tilde E_f(x)$ are $\gamma_{ij}(x; f), i \in \mathcal{D}, j \ge 1,$ and are given by: $$\gamma_{ij}(x; f)= 2^{-1 -R_i(x, j; f)}, \  \text{where} \ R_i(x, j; f):= Card\{k, 1 \le k \le j, \ d_k(x) = i\},$$ where $d_k(x), k\ge 1$, are the digits of $x$ in its $f$-expansion. Then, using the dimensions of the sets $\tilde E_f(x)$ as in \ref{tildeE}, we can distinguish between behaviors of digits for various numbers $x \in F(\eta)$, as the rates $\gamma_{ij}(x ; f)$ depend on the speed of convergence of $\frac{R_i(x, j; f)}{j}$ to the asymptotic frequency $\eta_i$, when $i \in \mathcal{D}$ and $j \to \infty$. In particular, this works for normal numbers $x \in \mathcal{N}(f)$, in which case $\eta_i = \mu_T(I_i), \ i \in \mathcal{D}$, with $\mu_T$ being the unique $T$-invariant absolutely continuous probability.

Assume now that the possible values for digits in $f$-expansion form an \textbf{infinite set} $\mathcal{D}$ (as in the continued fraction expansion). In this subtler case, we cannot always define $\tilde E_f(x)$ as in \ref{tildeE}, since we need condition (\ref{infinite}).
First, let the vector $\eta(\phi)$ depending on the potential $\phi$, $$\eta(\phi) = (\eta_1(\phi), \eta_2(\phi), \ldots) := (\mu_\phi(I_1), \mu_\phi(I_2), \ldots)$$ For an arbitrary quasinormal number in the $f$-expansion, $x \in \mathcal{Q}_f$, denote by $$R_i(x, j):= Card \{k, 1\le k \le j, d_k(x) = i\}, \  \text{for} \  i \ge 1, j \ge 1$$ 
For the infinite vector $\eta(\phi)$ introduced above, consider the set $F(\eta(\phi))$ of $\eta(\phi)$-quasinormal numbers in $f$-expansion, and let now a number $x \in F(\eta(\phi))$. 
Then we have $\lim\limits_{j \to \infty}\frac{R_i(x, j)}{j} = \eta_i(\phi) >0, \ i \ge 1$; hence for any $i \ge 1$, there exists a minimal integer $N_i(x) \ge 2$ such that the following condition is satisfied: \ $R_i(x, j) > j \cdot \eta_i(\phi)/2, \ \text{for every} \ j > N_i(x)$. \ 

 Consider now the infinite non-stationary \textbf{contraction vectors $\Psi_i(x), i \ge 1$, which depend on $x$},
$$
\Psi_i(x) = (\gamma_{i1}(x), \gamma_{i2}(x), \ldots), \ i \ge 1, \ \text{where}
$$
\begin{equation}\label{psix}
\gamma_{ij}(x) := \frac{\eta_i(\phi)}{N_i(x)} \cdot 2^{-R_i(x, j)}, \ i, j \ge 1
\end{equation}
Then condition (\ref{infinite}) is satisfied for these contraction rates. Indeed, $\gamma_{ij}(x) \le \frac 12, i, j \ge 1$ and, since $\eta_i(\phi)$ decreases to 0 when $i \to \infty$, there exists a constant $C \in (0,  \infty)$ s.t for all $i \ge 1$, 
$$
\sum\limits_{j \ge 1} \gamma_{ij} =  \sum\limits_{j=1}^{N_i(x)} \gamma_{ij} + \sum\limits_{j > N_i(x)} \gamma_{ij} \le \eta_i(\phi) + \eta_i(\phi) \sum\limits_{j > N_i(x)} \frac{1}{2^{j\cdot \eta_i(\phi)/2}} \le \eta_i(\phi)+ \frac{2\eta_i(\phi)}{1-2^{\eta_i(\phi)/2}}
 < C
$$
We can then construct as in Section 2, a fractal $\tilde E_f(x) \subset \mathbb R^d$ associated to contraction vectors $\Psi_i(x)$ and frequencies $\eta_i(\phi)$, for $i \ge 1$. Clearly this construction is non-stationary since $\gamma_{ij}(x)$ depend on $j$, and there are countably many $\Psi_i(x)$ depending on all digits of $x$. \ 
We show now that $ \text{dim}_\text{H}(\tilde E_f(x))$ varies continuously when $x$ ranges in $F(\eta(\phi))$. Also, it can be approximated with dimensions of finite constructions, which may be useful in computer estimates:

\begin{theorem}\label{f-tilde-dim}
 Consider a function $f$ and a potential $\phi$ as in the setting of \ref{f-tilde}; consider the equilibrium measure $\mu_\phi$ and let a number $x \in F(\eta(\phi))$, where $\eta(\phi) = (\mu_\phi(I_1), \mu_\phi(I_2), \ldots)$. Construct the associated fractal  $\tilde E_f(x)$, and denote $ \delta_{f, \phi}(x):= \text{dim}_\text{H}(\tilde E_f(x))$.

 \ \ a) Then $\delta_{f, \phi}(x)$ is the unique solution of the equation:
$$\sum\limits_{i = 1}^\infty \eta_i(\phi) \log \sum\limits_{j=1}^\infty  \big(\frac{\eta_i(\phi)}{N_i(x)} \cdot 2^{-R_i(x, j)}\big)^{\delta_{f, \phi}(x)} = 0$$

Then, the function $x \mapsto \delta_{f, \phi}(x)$ is continuous when $x$ varies in $F(\eta(\phi))$.

 \ \ b) Moreover for any $\delta>0$ there exist integers $k = k(\delta)$ and  $n(\delta)$ and constants $c_{\delta}, C_{\delta}>0$ so that, for every $n > n(\delta)$ there exists a Borel set $A(f; \delta, n) \subset [0, 1)$ satisfying: $$\mu_\phi(A(f; \delta, n)) \ge 1- C_\delta e^{-n \cdot c_\delta}, \ \text{and for every}  \ x \in A(f; \delta, n) \ \text{we have}$$  $$\delta_{f, \phi}(x) - \delta < \delta_{f, \phi}(x; k, n) \le \delta_{f, \phi}(x) + \delta,$$  where $\delta_{f, \phi}(x; k, n)$ is the unique solution of the equation in $t$ with finitely many terms,
$$\sum\limits_{i=1}^k \eta_i(\phi) \cdot \log \sum\limits_{j=1}^n \big(\frac{\eta_i(\phi)}{N_i(x)}\cdot 2^{-R_i(x, j)}\big)^t = 0$$ 
\end{theorem}

\begin{proof} 
\ \ a) The continuity of $\delta_{f, \phi}(\cdot)$ on $F(\eta(\phi))$ is proved similarly as in Corollary \ref{cortildeE}.

 \ \ b) For integers $k, n \ge 2$ let us introduce now the following notation: 
$$A(f; k, n):= \{x \in [0, 1), \  \frac{R_i(x, n')}{n'} > \eta_i(\phi)/2, \  1 \le i \le k, n' >n\} 
$$
Then it follows from  \cite{KPW} that for every $k \ge 2$ there exist constants $c_{k}, C_{k}>0$ such that: 
$$
\mu_\phi(A(f; k, n)) = \sum\limits_{n' >n} \mu_\phi\big(x \in [0, 1), \ \frac{R_i(x, n')}{n'} > \eta_i(\phi)/2, \ 1 \le i \le k\big) \ge 1- C_k e^{-n \cdot c_k}$$
For $x \in A(f; k, n)$ and for $1 \le i \le k$, we have $N_i(x) \le n$.
Notice that for $n$ large enough, $$\sum\limits_{j \ge 1} \gamma_{ij}(x)^t  = (1+o(n)) \mathop{\sum}\limits_{1 \le j \le n} \gamma_{ij}^t,$$  where the term $o(n)$ is independent of $t$.
Hence for any $\delta>0$ there exist integers $k=k(\delta)$ and $n(\delta)$ such that for every $n > n(\delta)$ and $x \in A(f; k, n)$, we have  the estimates $$\delta_{f, \phi}(x) - \delta \le \delta_{f, \phi}(x; k, n) \le \delta_{f, \phi}(x) + \delta,$$ where $\delta_{f, \phi}(x; k, n)$ is the unique solution of the equation with finitely many terms \ 
$\sum\limits_{i=1}^k \eta_i(\phi) \cdot \log \sum\limits_{j= 1}^n (\frac{\eta_i(\phi)}{N_i(x)}\cdot 2^{-R_i(x, j)})^t = 0$. Let us denote $A(f; k, n)$ by $A(f; \delta, n)$ as $k = k(\delta)$. 
Then the conclusion of the Theorem follows from the estimates obtained above.

\end{proof}

In the end, we remark that Theorem \ref{f-tilde-dim} applies also to the $T$-invariant absolutely continuous measure $\mu_T$, given as equilibrium measure of $- \log |T'|$. 
In particular, it applies to the \textbf{continued fraction expansion} with Gauss measure $\mu_G$, to study fractals $\tilde E_G(x)$ constructed from digits of normal numbers $x$, and to use them to distinguish between such numbers.

\

\

\  Eugen Mihailescu, \ \ Institute of Mathematics "Simion Stoilow" of the Romanian Academy, P.O Box 1-764, RO 014700, Bucharest, Romania.

Email: Eugen.Mihailescu\@@imar.ro \ \ \ \ \
Webpage: www.imar.ro/$\sim$mihailes

\

\ Mrinal Kanti Roychowdhury, \ \ Department of Mathematics,
University of Texas-Pan American,
1201 West University Drive, Edinburg, TX 78539-2999, USA.

Email: roychowdhurymk@utpa.edu \ \ \ \
Webpage: http://faculty.utpa.edu/roychowdhurymk

\end{document}